\documentclass{article}
\usepackage[a4paper, margin=2.5cm]{geometry}
\usepackage[leqno,intlimits]{amsmath}
\usepackage{amssymb}
\usepackage{amsthm}
\usepackage{hyperref}
\usepackage{tikz}
\numberwithin{equation}{section}

\newtheorem{tm}{Theorem}[section]
\newtheorem{cor}[tm]{Corollary}
\newtheorem{pr}[tm]{Proposition}
\newtheorem{rem}[tm]{Remark}
\newtheorem{lm}[tm]{Lemma}

\newcommand*{\Lk}{L^\text{bulk}}
\newcommand*{\Ly}{L^\text{bry}}
\newcommand*{\La}{L^\text{aux}}
\newcommand*{\vril}{\un\vr^{i\leftarrow}}
\newcommand*{\vrir}{\un\vr^{i\rightarrow}}
\newcommand*{\Zb}{\mathbb Z}
\newcommand*{\Kc}{\mathcal K}
\newcommand*{\om}{\omega}
\newcommand*{\un}[1]{\underline{#1}}
\newcommand*{\e}[1]{\text{\rm e}^{#1}}
\newcommand*{\be}{\beta}
\newcommand*{\vr}{\varrho}
\newcommand*{\te}{\theta}
\newcommand*{\Rb}{\mathbb R}
\DeclareMathOperator{\Eb}{{\mathbb E}}
\newcommand*{\ze}{\zeta}
\newcommand*{\vp}{\varphi}
\newcommand*{\di}{\,\text{d}}
\newcommand*{\hop}{\bigskip\noindent}

\begin{document}

\title{Semi-open boundary for random walking shocks}

\author{M\'arton Bal\'azs\footnotemark[1]\thanks{University of Bristol, UK. M.B.\ was partially supported by EPSRC's EP/R021449/1 and EP/W032112/1 Standard Grants.}}

\maketitle

\begin{abstract}
 Non-stationary time evolution of interacting particle systems is in general a rather difficult topic however, exceptional examples are known where hidden processes within the model make the description manageable. These hidden processes often take the form of a finite number of interacting random walks and in some cases, rather than being hidden, were very explicitly revealed as \emph{second class particles} associated with the model. The examples of asymmetric exclusion and exponential bricklayers model are known in this context, where such distributional structure was demonstrated in infinite volume. Here we find boundary mechanisms that save this remarkable structure in finite volume of the model with semi-open boundaries that let ordinary particles, but not the shocks, through. This finding also allows to characterise nontrivial two-species, non-reversible stationary distributions subject to our special boundary rates.
\end{abstract}

\section{Introduction}

Certain one-dimensional interacting particle systems are known to possess \emph{random walking shocks}, special distributions that are mixtures of product distributions with piecewise constant density values. While these are not stationary distributions of the system, their time evolution is described in a relatively simple way by considering the locations where the density jumps. These locations turn out to perform interacting simple random walks, and in some cases are known to be microscopically characterised by the presence of a \emph{second class particle}. On the unconstrained integer lattice \(\Zb\), the phenomenon is described in Bal\'azs, Farkas, Kov\'acs, R\'akos \cite{rwshscp}, and a further exhaustive search within the nearest-neighbour zero range family has been carried out in Bal\'azs, Duffy, Pantelli \cite{b_d_p_qzrp_rwsh_s}. \cite{rwshscp} will be the starting point for our work here.

The generic structure of these measures is product in space with one free density parameter on, say, the left of the first shock. The density values on the right of subsequent shocks and the jump rates of the random walking shocks are then fixed by algebraic relations. The three families of processes with such phenomenon we consider are
\begin{itemize}
 \item Asymmetric Simple Exclusion Process (ASEP), first treated by Belitsky and Sch\"utz \cite{qse}, the second class particle added in later by R\'akos and Sch\"utz \cite{Rakos2004} and Bal\'azs, Farkas, Kov\'acs, R\'akos \cite{rwshscp}. The shocks jump both to the left and to the right.
 \item Zero range-type processes, exhaustively investigated for random walking shocks by Bal\'azs, Duffy, Pantelli \cite{b_d_p_qzrp_rwsh_s}. Only the totally asymmetric case works here, hence the shock only jumps to the right.
 \item Bricklayers process, considered by a series of papers, Bal\'azs \cite{valak,sokvalak} and Bal\'azs, Farkas, Kov\'acs, R\'akos \cite{rwshscp}. This is essentially two zero range generators put together, one jumping left and one to the right. Therefore the shock jumps to both left and right.
\end{itemize}

In this note we identify semi-open boundary mechanisms with which the above random walking shocks can be confined to a finite volume of \(\Zb\). Our boundary is open for ordinary particles, only reflecting shocks and the second class particles sitting inside. Finding such setup that is compatible with the random walking shocks dynamics was mentioned as an open problem in \cite{bel_schutz_selfd_shockdyn_18}. As a consequence we also characterise the stationary distributions of the confined model under our boundary mechanisms.

The asymmetry in the jump dynamics makes the process non-reversible. To keep this problem interesting, we need the shocks to be able to jump both left and right, otherwise they would accumulate on one end of the volume. This excludes zero range processes from the above list of models with random walking shocks, leaving the discussion to ASEP and bricklayers processes only. Our arguments extend, and therefore heavily rely on, \cite{rwshscp} and our notation is compatible with that article. 

Surprisingly simple random walks can appear in the time-evolution of complicated systems. We already mentioned known examples of conservative interacting particle systems \cite{qse,valak,sokvalak,rwshscp,b_d_p_qzrp_rwsh_s}. R\'akos and Sch\"utz \cite{Rakos2004} osberve this phenomenon in a more general two-species exclusion model, and they can also fit open boundary rates to keep the random walk structure of shocks. Belitsky and Sch\"utz \cite{beli_schutz_selfd_2comp_15} consider ASEP with second class particles and reflecting boundary condition for all species. They characterise reversible stationary measures via stochastic self-duality of the process. The authors also consider more than two species in this context in \cite{bel_schutz_selfd_shockdyn_18}, while Sch\"utz and Tabatabaei \cite{schtaba} generalise this to ASEP with internal degrees of freedom. De Masi, Ferrari and Gabrielli \cite{masi_fer_gabr_hidden_temp_kmp} used a structure of hidden random variables to explore stationary distributions in the Kipnis-Marchioro-Presutti model, and Giardin\`a, Redig and van Tol \cite{giar_redig_vtol_intert_prop_mixt} generalised this to further models. These hidden random variables appear to play a similar role to the random walkers in the previous works.

Sch\"utz \cite{schutz_rev_duality} introduced the notion of reverse duality to investigate the behaviour of open boundary ASEP in terms of an inhomogeneous instance with reflecting boundaries. Given the right boundary parameters, shock product measures are shown to evolve as the particles of this dual process, with reflecting boundaries. Our ASEP case is rather similar to this work however, our use of the second class particle to trace the shock locations represents a difference. 

Duality and quantum algebra methods play an important role in many results of the literature. These considerations even led to the introduction of seemingly unnatural models which, however, possess very convenient algebraic properties. Of the vast literature we mention the work of Carinci, Giardin\`a and Redig \cite{car_gia_red_sas_gen_asep} as an example introducing the ASEP(\(q,\,j\)) model, the duality properties of which were extended to the multi-species case by Kuan \cite{kuan_multi_aqj_18}.

Our work adds to the state-of-the-art by making the boundary open for ordinary particles but reflective for the shock-walkers, at the same time identifying these with the second class particle. This way these often hidden random variables become very explicit in the microscopic configuration itself. We do this for ASEP which has been much investigated before, but also for bricklayers about which much less is known. In particular integrable, duality, or quantum algebra structures are, to the best of our knowledge, completely missing for this model. In fact we work with elementary generator manipulations, though our calculations become tedious at places.

\section{The models and results}

We only consider the ASEP and the \emph{totally asymmetric exponential bricklayers process} (abbreviated BLP). We now recall their definitions from \cite{rwshscp}. Let \(I=\{0,\,1\}\) for ASEP and \(I=\Zb\) for BLP, this represents the set of possible local particle numbers per site. In \cite{rwshscp} the state space was \(I^\Zb\) (modulo spatial growth restrictions for BLP, to make sure the process can be constructed, see Bal\'azs, Rassoul-Agha, Sepp\"al\"ainen and Sethuraman \cite{exists} and the recent paper by Andjel, Armend\'ariz, Jara \cite{and_arm_jara_zrp_rapid}). Here we fix an integer \(K>0\) and restrict our volume to the discrete interval \(\Kc:\,=\{1\dots K\}\), hence our state space will be \(I^\Kc\). This introduces the boundary edges \((0,\,1)\) and \((K,\,K+1)\) where we need to specify the boundary mechanism that was not present in the setting of \cite{rwshscp}, additional to the bulk dynamics which will be unchanged. Later on we also address half-line models where only one of the two boundaries \((0,\,1)\) or \((K,\,K+1)\) is present, the other side goes on indefinitely.

We start by repeating the definitions of the bulk jump rates from \cite{rwshscp}. With
\begin{equation}
 \om^{i,j}_k:\,=\left\{\begin{aligned}
  &\om_k,&&\text{for }k\ne i,\,j,\\
  &\om_k-1,&&\text{for }k=i,\\
  &\om_k+1,&&\text{for }k=j
 \end{aligned}\right.\label{eq:jump}
\end{equation}
for a configuration \(\un\om=(\om_k)_{k\in\Kc}\in I^\Kc\), the bulk dynamics of either ASEP or BLP consists of right and left jumps
\[
 \begin{aligned}
  \un\om&\to\un\om^{k,\,k+1}\qquad&\text{with rate}\qquad&p(\om_k,\,\om_{k+1}),\qquad\text{and}\\
  \un\om&\to\un\om^{k+1,\,k}\qquad&\text{with rate}\qquad&q(\om_k,\,\om_{k+1})
 \end{aligned}
\]
for \(1\le k<K\). The jump rates are of the particular forms
\begin{align}
 p(y,\,z)&=p\cdot{\bf1}\{y=1,\,z=0\}\qquad&\text{and}\qquad q(y,\,z)&=q\cdot{\bf1}\{y=0,\,z=1\}\qquad&&\text{(ASEP),}\notag\\
 p(y,\,z)&=\e{\be(y-\frac12)}+\e{\be(-z-\frac12)}\qquad&\text{and}\qquad q(y,\,z)&=0\qquad&&\text{(BLP), }y,\,z\in\Zb.\label{eq:blppdef}
\end{align}

For ASEP we assume \(p+q=1\). Our notation will be natural for the case \(p>q\) but all our results hold for \(p<q\) as well. We also repeat from \cite{rwshscp} that the i.i.d.\ Bernoulli(\(\vr\)) distribution for any \(0\le\vr\le1\) is stationary for ASEP if we replace \(\Kc\) by the infinite spatial volume \(\Zb\). Similarly, the product measure of site-marginals
\begin{equation}
 \mu^\te(z)=\frac{\e{\te^2\!/2\beta}}{Z(\te)}\cdot\e{-\frac\beta2\cdot\bigl(z-\frac\te\beta\bigr)^2}=\frac{\e{-\frac\be2z^2+\te z}}{Z(\te)}\qquad z\in I=\Zb\label{eq:gauss}
\end{equation}
for any \(\te\in\Rb\) is stationary for BLP on the infinite volume. Here \(Z(\te)\) is the normalisation constant which is related to the Jacobi theta function and has no closed form. We have the remarkable properties
\[
 Z(\te-\beta)=\e{\beta/2-\te}\cdot Z(\te)
\]
and with the expectation w.r.t.\ \(\mu^\te\),
\begin{equation}
 \vr(\te):\,=\Eb^{\mu^\te}(\om),\label{eq:rt}
\end{equation}
\(\rho(\te+\beta)=\rho(\te)+1\) holds. This density function \(\vr(\te)\) is shown in \cite{exists} to be a strictly monotone bijection from \(\Rb\) to \(\Rb\) for BLP. For its inverse, we therefore have
\begin{equation}
 \te(\vr+1)=\te(\vr)+\be.\label{eq:teshi}
\end{equation}
We also note that, with
\[
 f(y):\,=\e{\be(y-\frac12)}
\]
we have
\begin{equation}
 f(y)\cdot f(1-y)=1\qquad\forall y\in\Zb,\label{eq:fprod}
\end{equation}
\begin{equation}
 f(y+m)=\e{m\be}f(y)\qquad\forall y,\,m\in\Zb,\label{eq:bef}
\end{equation}
\begin{equation}
 p(y,\,z)=f(y)+f(-z),\label{eq:pf}
\end{equation}
and due to \(\be\bigl(\pm y-\frac12\bigr)-\frac\be2\cdot\bigl(y-\frac\te\be\bigr)^2=\pm\te-\frac\be2\cdot\bigl(y\mp1-\frac\te\be\bigr)^2\),
\begin{equation}
 f(\pm y)\cdot\mu^\te(y)=\e{\pm\te}\mu^\te(y\mp1),\qquad
 \Eb^{\mu^\te}f(\pm\om)=\sum_yf(\pm y)\cdot\mu^\te(y)=\e{\pm\te}\sum_y\mu^\te(y\mp1)=\e{\pm\te}.\label{eq:ef}
\end{equation}

With a slight abuse of notation, from here on we commonly denote by \(\mu^\vr\) the above respective distributions on \(I\) with mean \(\vr\in[0,\,1]\) for ASEP and \(\vr\in\Rb\) for BLP. Here we slightly deviate from the notation of \cite{rwshscp} as for the two models we cover it is sometimes simpler to use densities than the parameter \(\te\).

Next we turn to the joint evolution of a coupled pair of processes \((\un\om,\,\un\ze)\). We first describe the coupled dynamics in the bulk of \(\Kc\); the primary aim of this paper will be to fix this later for the boundaries in a way that preserves the nice distributional structures from \cite{rwshscp}. For both ASEP and BLP, we easily check
\[
 \begin{aligned}
  p(z+1,\,y)&\ge p(z,\,y),\qquad&p(y,\,z+1)&\le p(y,\,z)\\
  q(z+1,\,y)&\le q(z,\,y),\qquad&q(y,\,z+1)&\ge q(y,\,z)
 \end{aligned}
\]
which allows the \emph{basic coupling} to do its magic and preserve the coordinatewise inequality \(\om_i(t)\le\ze_i(t)\), \(\forall i\in\Kc\) if this was valid initially. The bulk basic coupling generator is cited from \cite{rwshscp}; for a function \(\vp\,:\,I^\Kc\to\Rb\) and states \(\om_i\le\ze_i\) (\(1\le i\le K\)),
\begin{equation}
 \begin{aligned}
  (\Lk\vp)(\un\om,\,\un\ze)=\sum_{k=1}^{K-1}\Bigl\{&p(\om_k,\,\ze_{k+1})\cdot\bigl[\vp(\un\om^{k,k+1},\,\un\ze^{k,k+1})-\vp(\un\om,\,\un\ze)\bigr]\\
  +\bigl[&p(\om_k,\,\om_{k+1})-p(\om_k,\,\ze_{k+1})\bigr]\cdot\bigl[\vp(\un\om^{k,k+1},\,\un\ze)-\vp(\un\om,\,\un\ze)\bigr]\\
  +\bigl[&p(\ze_k,\,\ze_{k+1})-p(\om_k,\,\ze_{k+1})\bigr]\cdot\bigl[\vp(\un\om,\,\un\ze^{k,k+1})-\vp(\un\om,\,\un\ze)\bigr]\\
  +\phantom{\bigl[}&q(\ze_k,\,\om_{k+1})\cdot\bigl[\vp(\un\om^{k+1,k},\,\un\ze^{k+1,k})-\vp(\un\om,\,\un\ze)\bigr]\\
  +\bigl[&q(\ze_k,\,\ze_{k+1})-q(\ze_k,\,\om_{k+1})\bigr]\cdot\bigl[\vp(\un\om,\,\un\ze^{k+1,k})-\vp(\un\om,\,\un\ze)\bigr]\\
  +\bigl[&q(\om_k,\,\om_{k+1})-q(\ze_k,\,\om_{k+1})\bigr]\cdot\bigl[\vp(\un\om^{k+1,k},\,\un\ze)-\vp(\un\om,\,\un\ze)\bigr]\Bigr\}.
 \end{aligned}\label{eq:lbulkdef}
\end{equation}
The difference \(\ze_i-\om_i\) is the number of \emph{second class particles} at site \(i\), and its sum for all sites is easily checked to be a conserved quantity as far as this bulk dynamics is concerned.

To describe the boundary dynamics, we first define the jumps (to be used for \(i=1\) and \(i=K\) only)
\begin{equation}
 \om^{i\pm}_k:\,=\left\{\begin{aligned}
  &\om_k,&&\text{for }k\ne i,\\
  &\om_k\pm1,&&\text{for }k=i
 \end{aligned}\right.\label{eq:1sp}
\end{equation}
and the boundary rates \(p_0(\cdot,\,\cdot)\), \(q_0(\cdot,\,\cdot)\), \(p_K(\cdot,\,\cdot)\), \(q_K(\cdot,\,\cdot)\). As opposed to the unindexed \(p\) and \(q\) rates, the two variables these functions depend on will be the particle occupations \((\om_1,\,\ze_1)\) and \((\om_K,\,\ze_K)\) for the coupled pair \((\un\om,\,\un\ze)\) on the extremal sites of \(\Kc\).

We are looking for boundary dynamics that is open for first class particles but reflects second class particles in order to keep the number of the latter conserved. That leaves us with the natural choices for moves as
\begin{equation}
 \begin{aligned}
  (\un\om,\,\un\ze)&\to(\un\om^{1+},\,\un\ze^{1+})\qquad&\text{with rate}\qquad&p_0(\om_1,\,\ze_1),\\
  (\un\om,\,\un\ze)&\to(\un\om^{1-},\,\un\ze^{1-})\qquad&\text{with rate}\qquad&q_0(\om_1,\,\ze_1),\\
  (\un\om,\,\un\ze)&\to(\un\om^{K-},\,\un\ze^{K-})\qquad&\text{with rate}\qquad&p_K(\om_K,\,\ze_K),\qquad\text{and}\\
  (\un\om,\,\un\ze)&\to(\un\om^{K+},\,\un\ze^{K+})\qquad&\text{with rate}\qquad&q_K(\om_K,\,\ze_K).
 \end{aligned}\label{eq:bdyr}
\end{equation}
The boundary generator is built of these jumps:
\begin{equation}
 \begin{aligned}
  (\Ly\vp)(\un\om,\,\un\ze)&=p_0(\om_1,\,\ze_1)\bigl[\vp(\un\om^{1+},\,\un\ze^{1+})-\vp(\un\om,\,\un\ze)\bigr]\\
  &+q_0(\om_1,\,\ze_1)\bigl[\vp(\un\om^{1-},\,\un\ze^{1-})-\vp(\un\om,\,\un\ze)\bigr]\\
  &+p_K(\om_K,\,\ze_K)\bigl[\vp(\un\om^{K-},\,\un\ze^{K-})-\vp(\un\om,\,\un\ze)\bigr]\\
  &+q_K(\om_K,\,\ze_K)\bigl[\vp(\un\om^{K+},\,\un\ze^{K+})-\vp(\un\om,\,\un\ze)\bigr]
 \end{aligned}\label{eq:lbdydef}
\end{equation}
and the object of our study is the Markov process on \(I^\Kc\) defined by the generator
\begin{equation}
 L:\,=\Lk+\Ly.\label{eq:ldef}
\end{equation}

We notice that the boundary rates might depend on both the \(\om\) and \(\ze\) values, which can stop either of the marginals \(\un\om(t)\) or \(\un\ze(t)\) being a Markov chain on its own. In particular, to keep the state space, the ASEP case only allows \((0,\,0)\) inside the \(p_0\) and \(q_K\) rates, and \((1,\,1)\) in the \(q_0\) and \(p_K\) rates for these to be non-zero.

\subsection{Confined random walking shocks}

Next we introduce the notation for the measures that are needed to state the main result. We parametrise these with the density \(\vr\) which is the only difference between our treatment here and that of \cite{rwshscp}. For a site \(i\in\Kc\), the marginal distribution of the pair \((\om_i,\,\ze_i)\) will be denoted by
\[
 \nu^{\vr_i,m_i}(y,\,z)=\left\{
  \begin{aligned}
   &\hat\mu^{\vr_i,m_i}(y),&&\text{if }z=y+m_i,\\
   &0,&&\text{if }z\ne y+m_i.
  \end{aligned}
 \right.
\]
Notice that this definition includes \(\hat\mu^{\cdot,1}(1)=0\) in the case of ASEP as this would require value 2 in the second variable of \(\nu^{\cdot,1}\). The density values \(\vr_i\) can be in \([0,\,1]\) for ASEP and in all of \(\Rb\) for BLP, while the integers \(m_i\ge0\) describe the number of second class particles at site \(i\). This latter is not part of the distribution, instead, both \(\vr_i\) and \(m_i\) are given parameters of the distribution. The measures \(\hat\mu^{\vr_i,m_i}\) describe the distribution of the smaller marginal \(\om_i\) when \(m_i\) second class particles are present at site \(i\) and will be specified further below. The distribution of interest for the coupled pair \((\un\om,\,\un\ze)\) will be the product of these marginals:
\begin{equation}
 \un\nu^{\un\vr,\un m}:\,=\bigotimes_{i\in\Kc}\nu^{\vr_i,m_i}\label{eq:nurm}
\end{equation}
with the parameter vectors \(\un\vr=(\vr_i)_{i\in\Kc}\) and \(\un m=(m_i)_{i\in\Kc}\).

The result of our work is that Theorem 2 (and hence Theorem 1 as well) of \cite{rwshscp} can be repeated for the finite volume \(\Kc\) with properly chosen boundary rates \eqref{eq:bdyr} that reflect the random walking shocks and second class particles. We extend the notation \eqref{eq:jump} to \(\un m\) as
\[
 m^{i,j}_k:\,=\left\{\begin{aligned}
  &m_k,&&\text{for }k\ne i,\,j,\\
  &m_k-1,&&\text{for }k=i,\\
  &m_k+1,&&\text{for }k=j,
 \end{aligned}\right.
\]
while \(\vrir\) and \(\vril\) will be defined below. Following standard notation, \(\un\nu^{\un\vr,\un m}S(t)\) denotes the distribution of the coupled pair \((\un\om(t),\,\un\ze(t))\) when started from distribution \(\un\nu^{\un\vr,\un m}\) and evolving according to generator \(L=\Lk+\Ly\).

\begin{figure}[ht]
 \begin{center}
  \begin{tikzpicture}
   \draw[->](-0.25,2.5)--(8.25,2.5)node[below right]{\small\(\Zb\)};
   \foreach\x in{0,0.5,...,8}
   \draw(\x,2.4)--(\x,2.6);
   \draw(3,2.4)node[below]{\tiny\(i\)};
   \draw(3.5,2.4)node[below]{\tiny\(i+1\)};
   \filldraw(0,2.8)circle[radius=0.1];
   \filldraw(1.5,2.8)circle[radius=0.1];
   \filldraw(2.5,2.8)circle[radius=0.1];
   \filldraw(4.5,2.8)circle[radius=0.1];
   \filldraw(5.5,2.8)circle[radius=0.1];
   \filldraw(7,2.8)circle[radius=0.1];

   \filldraw(0,3.1)circle[radius=0.1];
   \filldraw(1.5,3.1)circle[radius=0.1];
   \filldraw(2.5,3.1)circle[radius=0.1];
   \filldraw(3,3.1)circle[radius=0.1];
   \filldraw(4.5,3.1)circle[radius=0.1];
   \filldraw(5,3.1)circle[radius=0.1];
   \filldraw(5.5,3.1)circle[radius=0.1];
   \filldraw(7,3.1)circle[radius=0.1];

   \foreach\x in{0,0.5,...,2.5}
   \draw(\x-0.15,3.4)--(\x+0.15,3.4);
   \foreach\x in{3,3.5,...,4.5}
   \draw(\x-0.15,3.6)--(\x+0.15,3.6);
   \draw(3,3.6)node[above]{\tiny\(\vr_i\)};
   \draw(3.5,3.6)node[above]{\tiny\(\vr_{i+1}\)};
   \foreach\x in{5,5.5,...,8}
   \draw(\x-0.15,3.75)--(\x+0.15,3.75);

   \draw[->](-0.25,0.5)--(8.25,0.5)node[below right]{\small\(\Zb\)};
   \foreach\x in{0,0.5,...,8}
   \draw(\x,0.4)--(\x,0.6);
   \draw(3,0.4)node[below]{\tiny\(i\)};
   \draw(3.5,0.4)node[below]{\tiny\(i+1\)};
   \filldraw(0,0.8)circle[radius=0.1];
   \filldraw(1.5,0.8)circle[radius=0.1];
   \filldraw(2.5,0.8)circle[radius=0.1];
   \filldraw(4.5,0.8)circle[radius=0.1];
   \filldraw(5.5,0.8)circle[radius=0.1];
   \filldraw(7,0.8)circle[radius=0.1];

   \filldraw(0,1.1)circle[radius=0.1];
   \filldraw(1.5,1.1)circle[radius=0.1];
   \filldraw(2.5,1.1)circle[radius=0.1];
   \filldraw(3.5,1.1)circle[radius=0.1];
   \filldraw(4.5,1.1)circle[radius=0.1];
   \filldraw(5,1.1)circle[radius=0.1];
   \filldraw(5.5,1.1)circle[radius=0.1];
   \filldraw(7,1.1)circle[radius=0.1];

   \foreach\x in{0,0.5,...,3}
   \draw(\x-0.15,1.4)--(\x+0.15,1.4);
   \draw(3,1.4)node[above]{\tiny\(\vr_i\)};
   \foreach\x in{3.5,4,...,4.5}
   \draw(\x-0.15,1.6)--(\x+0.15,1.6);
   \draw(3.5,1.6)node[above]{\tiny\(\vr_{i+1}\)};
   \foreach\x in{5,5.5,...,8}
   \draw(\x-0.15,1.75)--(\x+0.15,1.75);
   \draw[densely dotted,->](3,1.1)--(3.35,1.1);
   \draw[densely dotted,->](3,1.9)--(3,1.7);
  \end{tikzpicture}
 \end{center}
 \caption{Right step of an ASEP second class particle with the change in density that follows}\label{fig:aseprst}
\end{figure}
\begin{figure}[ht]
 \begin{center}
  \begin{tikzpicture}
   \draw[->](-0.25,2.5)--(8.25,2.5)node[below right]{\small\(\Zb\)};
   \foreach\x in{0,0.5,...,8}
   \draw(\x,2.4)--(\x,2.6);
   \draw(2.5,2.4)node[below]{\tiny\(i\)};
   \draw(3,2.4)node[below]{\tiny\(i+1\)};
   \filldraw(0,2.8)circle[radius=0.1];
   \filldraw(1.5,2.8)circle[radius=0.1];
   \filldraw(2.5,2.8)circle[radius=0.1];
   \filldraw(4.5,2.8)circle[radius=0.1];
   \filldraw(5.5,2.8)circle[radius=0.1];
   \filldraw(7,2.8)circle[radius=0.1];

   \filldraw(0,3.1)circle[radius=0.1];
   \filldraw(1.5,3.1)circle[radius=0.1];
   \filldraw(2.5,3.1)circle[radius=0.1];
   \filldraw(3,3.1)circle[radius=0.1];
   \filldraw(4.5,3.1)circle[radius=0.1];
   \filldraw(5,3.1)circle[radius=0.1];
   \filldraw(5.5,3.1)circle[radius=0.1];
   \filldraw(7,3.1)circle[radius=0.1];

   \foreach\x in{0,0.5,...,2.5}
   \draw(\x-0.15,3.4)--(\x+0.15,3.4);
   \draw(2.5,3.4)node[above]{\tiny\(\vr_i\)};
   \foreach\x in{3,3.5,...,4.5}
   \draw(\x-0.15,3.6)--(\x+0.15,3.6);
   \draw(3,3.6)node[above]{\tiny\(\vr_{i+1}\)};
   \foreach\x in{5,5.5,...,8}
   \draw(\x-0.15,3.75)--(\x+0.15,3.75);

   \draw[->](-0.25,0.5)--(8.25,0.5)node[below right]{\small\(\Zb\)};
   \foreach\x in{0,0.5,...,8}
   \draw(\x,0.4)--(\x,0.6);
   \draw(2.5,0.4)node[below]{\tiny\(i\)};
   \draw(3,0.4)node[below]{\tiny\(i+1\)};
   \filldraw(0,0.8)circle[radius=0.1];
   \filldraw(1.5,0.8)circle[radius=0.1];
   \filldraw(3,0.8)circle[radius=0.1];
   \filldraw(4.5,0.8)circle[radius=0.1];
   \filldraw(5.5,0.8)circle[radius=0.1];
   \filldraw(7,0.8)circle[radius=0.1];

   \filldraw(0,1.1)circle[radius=0.1];
   \filldraw(1.5,1.1)circle[radius=0.1];
   \filldraw(2.5,1.1)circle[radius=0.1];
   \filldraw(3,1.1)circle[radius=0.1];
   \filldraw(4.5,1.1)circle[radius=0.1];
   \filldraw(5,1.1)circle[radius=0.1];
   \filldraw(5.5,1.1)circle[radius=0.1];
   \filldraw(7,1.1)circle[radius=0.1];

   \foreach\x in{0,0.5,...,2}
   \draw(\x-0.15,1.4)--(\x+0.15,1.4);
   \foreach\x in{2.5,3,...,4.5}
   \draw(\x-0.15,1.6)--(\x+0.15,1.6);
   \draw(2.5,1.6)node[above]{\tiny\(\vr_i\)};
   \draw(3,1.6)node[above]{\tiny\(\vr_{i+1}\)};
   \foreach\x in{5,5.5,...,8}
   \draw(\x-0.15,1.75)--(\x+0.15,1.75);
   \draw[densely dotted,->](2.5,0.8)--(2.85,0.8);
   \draw[densely dotted,->](2.5,1.35)--(2.5,1.55);
  \end{tikzpicture}
 \end{center}
 \caption{Left step of an ASEP second class particle with the change in density that follows}\label{fig:aseplst}
\end{figure}

\begin{tm}\label{tm:mush}
 The identity
 \begin{equation}
  \begin{aligned}
   \frac{\di}{\di t}\un\nu^{\un\vr,\un m}S(t)\Bigr|_{t=0}=\sum_{i=1}^{K-1}&P(m_i,\,m_{i+1},\,\vr_i,\,\vr_{i+1})\cdot\bigl[\un\nu^{\vrir,\un m^{i,i+1}}-\un\nu^{\un\vr,\un m}\bigr]\\
   +\,&Q(m_i,\,m_{i+1},\,\vr_i,\,\vr_{i+1})\cdot\bigl[\un\nu^{\vril,\un m^{i+1,i}}-\un\nu^{\un\vr,\un m}\bigr]
  \end{aligned}\label{eq:mush}
 \end{equation}
 holds in the following special cases with the following boundary rates \(p_0\), \(q_0\), \(p_K\), \(q_K\) and parameters \(\vrir\), \(\vril\), \(P(m_i,\,m_{i+1},\,\vr_i,\,\vr_{i+1})\) and \(Q(m_i,\,m_{i+1},\,\vr_i,\,\vr_{i+1})\):
 \begin{itemize}
  \item For the ASEP, \(m_i=0\) or \(1\),
  \begin{equation}
   \begin{aligned}
    p_0(0,\,0)&=(p+c)\vr_1,&\qquad q_0(1,\,1)&=(q+c)(1-\vr_1),\\
    p_K(1,\,1)&=(p+d)(1-\vr_K),&\qquad q_K(0,\,0)&=(q+d)\vr_K
   \end{aligned}\label{eq:asepbry}
  \end{equation}
 and zero in all other cases, with arbitrary constants \(c\) and \(d\) in \([-\min(p,\,q),\,\infty)\), the relation
  \begin{equation}
   \frac{\vr_i(1-\vr_{i-1})}{\vr_{i-1}(1-\vr_i)}=\left\{
    \begin{aligned}
     &1,&&\text{if }m_i=0,\\
     &\frac pq,&&\text{if }m_i=1
    \end{aligned}
   \right.\qquad(1<i\le K)\label{eq:asepsokfelt}
  \end{equation}
  holds between the densities and the asymmetry, and
  \begin{equation}
   \hat\mu^{\vr,m}(0)=1-\hat\mu^{\vr,m}(1)=\left\{
    \begin{aligned}
     &1-\vr,&&\text{if }m=0,\\
     &1,&&\text{if }m=1.
    \end{aligned}
   \right.\label{eq:asepsokmuhat}
  \end{equation}
  In this case
  \begin{align}
   \vr_j^{i\rightarrow}&=\left\{
    \begin{aligned}
     &\vr_j,&&\text{for }j\ne i,\\
     &\frac{q\vr_i}{p-(p-q)\vr_i},&&\text{for }j=i,
    \end{aligned}
   \right.\label{eq:asepvrj}\\
   \vr_j^{i\leftarrow}&=\left\{
    \begin{aligned}
     &\vr_j,&&\text{for }j\ne i,\\
     &\frac{p\vr_i}{q+(p-q)\vr_i},&&\text{for }j=i\text{, and}
    \end{aligned}
   \right.\notag\\
   P(m_i,\,m_{i+1},\,\vr_i,\,\vr_{i+1})&=m_i(1-m_{i+1})\cdot\bigl[(1-\vr_{i+1})p+\vr_{i+1}q\bigr],\label{eq:asepshrj}\\
   Q(m_i,\,m_{i+1},\,\vr_i,\,\vr_{i+1})&=(1-m_i)m_{i+1}\cdot\bigl[(1-\vr_i)q+\vr_ip\bigr].\label{eq:asepshrb}
  \end{align}
  \item For BLP, with arbitrary non-negative \(q^+_0(x,\,x+m)\) and \(q^+_K(x,\,x+m)\) (for \(x\in\Zb\), \(m\in\Zb_{\ge0}\))
  \[
   \begin{aligned}
    p^+_0(x,\,x+m):&=q^+_0(x+1,\,x+m+1)\cdot\e{\te_1}f(-x),\\
    p^+_K(x,\,x+m):&=q^+_K(x-1,\,x+m-1)\cdot\e{-\te_K}f(x),
   \end{aligned}
  \]
  the boundary rates are
  \begin{equation}
   \begin{aligned}
    p_0(\om_1,\,\om_1+m_1)&=\e{\te_0}+f(-\om_1)+p^+_0(\om_1,\,\om_1+m_1),\\
    p_K(\om_K,\,\om_K+m_K)&=f(\om_K+m_K)+\e{-\te_{K+1}}+p^+_K(\om_K,\,\om_K+m_K),\\
    q_0(\om_1,\,\om_1+m_1)&=q^+_0(\om_1,\,\om_1+m_1),\\
    q_K(\om_K,\,\om_K+m_K)&=q^+_K(\om_K,\,\om_K+m_K),
   \end{aligned}\label{eq:blpbry}
  \end{equation}
  where \(\te_0:\,=\te_1+\be m_1\) and \(\te_{K+1}:\,=\te_K\), the relations
  \begin{equation}
   \vr_{i-1}-\vr_i=m_i,\qquad\text{equivalently,}\qquad\te_{i-1}-\te_i=\be m_i \label{eq:blpsokfelt}
  \end{equation}
  hold between the parameters, and
  \begin{equation}
   \hat\mu^{\vr,m}(y)=\mu^\vr(y)\label{eq:blpsokmuhat}
  \end{equation}
  regardless of \(m\). In this case we have
  \begin{align}
   \vr_j^{i\rightarrow}&=\left\{
    \begin{aligned}
     &\vr_j,&&\text{for }j\ne i,\\
     &\vr_i+1,&&\text{for }j=i,
    \end{aligned}
   \right.\notag\\
   \vr_j^{i\leftarrow}&=\left\{
    \begin{aligned}
     &\vr_j,&&\text{for }j\ne i,\\
     &\vr_i-1,&&\text{for }j=i,
    \end{aligned}
   \right.\notag\\
   P(m_i,\,m_{i+1},\,\vr_i,\,\vr_{i+1})&=\e{\te(\vr_i+m_i)}-\e{\te(\vr_i)}\qquad\text{and}\label{eq:blpshrj}\\
   Q(m_i,\,m_{i+1},\,\vr_i,\,\vr_{i+1})&=\e{-\te(\vr_{i+1})}-\e{-\te(\vr_{i+1}+m_{i+1})}.\label{eq:blpshrb}
  \end{align}
 \end{itemize}
\end{tm}
See Figures \ref{fig:aseprst} and \ref{fig:aseplst} for illustration of the ASEP bulk dynamics. Notice that for \eqref{eq:asepvrj} indeed \(\vr_{i-1}=\frac{q\vr_i}{p-(p-q)\vr_i}\) whenever \(m_i=1\) and \(i>1\), but we haven't yet defined the value \(\vr_0\) when \(i=1\). Similarly, \(\vr_{i+1}=\frac{p\vr_i}{q+(p-q)\vr_i}\) whenever \(m_{i+1}=1\). The notation \(\vrir\) and \(\vril\) replace \(\un\vr^{i,i+1}\) \(\un\vr^{i+1,i}\) of \cite{rwshscp} for exactly the same reason. For BLP above we used \(\te(\vr)\), the inverse of the function \(\vr(\te)\) from \eqref{eq:rt}.

The interpretation with random walking shocks and second class particles is exactly as in \cite{rwshscp}, except that the walks take place in the finite volume \(\Kc\) and, as opposed to first class particles, cannot jump out of its boundaries. Briefly, second class particles see two different densities on their two sides, they perform simple random walks in their respective local shocks, interacting via their rank-dependent jump rates and, in ASEP, the exclusion rule as well. The only novelty from \cite{rwshscp} is that the construction survives the truncation to a finite volume if we impose the respective boundary rates \eqref{eq:asepbry} and \eqref{eq:blpbry}.

Notice that the ASEP boundaries can only move when there is no second class particle present at respective sites 1 and \(K\). The cases \(c=0\) or \(d=0\) correspond to Liggett's condition \cite{ligg_ergod_asep} that would make the Bernoulli product measures with the respective reservoir densities \(\vr_0=\vr_1\) and \(\vr_{K+1}=\vr_K\) stationary for a single-species boundaried ASEP. The parts with the constants \(c\) and \(d\) represent a reversible boundary mechanism with respect to the same Bernoulli(\(\vr_0\)) or Bernoulli(\(\vr_{K+1}\)) distribution, adding this will not harm the stationary structure.

The BLP boundary rates consist of two parts. The fixed parts in \(p_0\) and \(p_K\) represent the larger, between the coupled pair \((\un\om,\,\un\ze)\), of the two expected BLP boundary jump rates averaged out at the hypothetical sites \(0\) and \(K+1\). Notice that this would be the stationary boundary for a single-specie BLP system akin to Liggett's condition of ASEP. However, as in the ASEP case, we are allowed the addition of a reversible boundary dynamics with arbitrary non-negative \(q^+_0\) and \(q^+_K\), with the corresponding \(p^+_0\) and \(p^+_K\) counterparts that make this part of the dynamics reversible. Reversibility of these rates is readily checked by \eqref{eq:ef} which provides the detailed balance equations
\[
 \begin{aligned}
  p^+_0(\om_1,\,\om_1+m_1)\mu^{\te_1}(\om_1)&=q^+_0(\om_1+1,\,\om_1+m_1+1)\cdot\e{\te_1}f(-\om_1)\mu^{\te_1}(\om_1)\\
  &=q^+_0(\om_1+1,\,\om_1+m_1+1)\mu^{\te_1}(\om_1+1),\\
  p^+_K(\om_K,\,\om_K+m_K)\mu^{\te_K}(\om_K)&=q^+_K(\om_K-1,\,\om_K+m_K-1)\cdot\e{-\te_K}f(\om_K)\mu^{\te_K}(\om_K)\\
  &=q^+_K(\om_K-1,\,\om_K+m_K-1)\mu^{\te_K}(\om_K-1).
 \end{aligned}
\]

As we shall see, all our proofs treat the left and the right boundaries separately, and compare this finite volume case to the infinite volume result of \cite{rwshscp}. Hence all our arguments remain valid for half-sided models of the following form:
\begin{itemize}
 \item take the sum in \(\Lk\) of \eqref{eq:lbulkdef} from \(k=1\) to \(k=\infty\) (respectively, \(k=-\infty\) to \(k=K-1\)),
 \item only write the first (resepctively, last) two lines of \(\Ly\) in \eqref{eq:lbdydef}.
\end{itemize}
We call this the half-line model.
\begin{cor}
 Theorem \ref{tm:mush} remains valid for the half-line model with one of the boundaries of the sum in \eqref{eq:mush} replaced by \(\infty\) accordingly, and removing the respective boundary mechanisms where the model doesn't end in space.
\end{cor}

\subsection{ASEP in half line}

For the next result, define auxiliary parameters \(m_{K+1}=0\), \(\vr_0\) and \(\vr_{K+1}\) that satisfy
\begin{equation}
 \begin{aligned}
  \frac{\vr_1(1-\vr_0)}{\vr_0(1-\vr_1)}&=\left\{
   \begin{aligned}
    &1,&&\text{if }m_1=0,\\
    &\frac pq,&&\text{if }m_1=1,
   \end{aligned}
  \right.\qquad&\vr_{K+1}&=\vr_K\qquad\text{(ASEP),}\\
  \vr_0-\vr_1&=m_1,&\vr_{K+1}&=\vr_K\qquad\text{(BLP),}
 \end{aligned}\label{eq:rnk}
\end{equation}
c.f.\ \eqref{eq:asepsokfelt} and \eqref{eq:blpsokfelt}. The total number of second class particles will be denoted \(M:\,=\sum_{i=1}^Km_i\) and notice that by recursive application of \eqref{eq:asepsokfelt}, the \(\ell^\text{th}\) ASEP second class particle (\(1\le\ell\le M\)) from the left has densities
\[
 \frac{\bigl(\frac pq\bigr)^{\ell-1}\cdot\vr_0}{1-\vr_0+\bigl(\frac pq\bigr)^{\ell-1}\cdot\vr_0}
 \qquad\text{and}\qquad
 \frac{\bigl(\frac pq\bigr)^\ell\cdot\vr_0}{1-\vr_0+\bigl(\frac pq\bigr)^\ell\cdot\vr_0}
\]
on the left neighbouring site and at its current site, and has right jump rate \eqref{eq:asepshrj}
\begin{equation}
 P^\ell=
 \frac{1-\vr_0}{1-\vr_0+\bigl(\frac pq\bigr)^\ell\cdot\vr_0}\cdot p
 +\frac{\bigl(\frac pq\bigr)^\ell\cdot\vr_0}{1-\vr_0+\bigl(\frac pq\bigr)^\ell\cdot\vr_0}\cdot q
 =p\cdot\frac{1-\vr_0+\bigl(\frac pq\bigr)^{\ell-1}\cdot\vr_0}{1-\vr_0+\bigl(\frac pq\bigr)^\ell\cdot\vr_0}\label{eq:ellthr}
\end{equation}
as long as it is not at the rightmost site of the volume and there is no other second class particle right next to it, and left jump rate \eqref{eq:asepshrb}
\begin{equation}
 Q^\ell=
 \frac{1-\vr_0}{1-\vr_0+\bigl(\frac pq\bigr)^{\ell-1}\cdot\vr_0}\cdot q
 +\frac{\bigl(\frac pq\bigr)^{\ell-1}\cdot\vr_0}{1-\vr_0+\bigl(\frac pq\bigr)^{\ell-1}\cdot\vr_0}\cdot p
 =q\cdot\frac{1-\vr_0+\bigl(\frac pq\bigr)^\ell\cdot\vr_0}{1-\vr_0+\bigl(\frac pq\bigr)^{\ell-1}\cdot\vr_0}\label{eq:ellthl}
\end{equation}
as long as it is not at the leftmost site of the volume and there is no other second class particle at its left neighbouring site. This ordered point of view will be useful for ASEP due to the exclusion component of the second class particles dynamics.

Our result yields stationary distributions of the multi-species models under our special boundary conditions. To explore this, we fist concentrate on the process described by \eqref{eq:mush} of second class particles. We start with the half-line case of ASEP with left boundary only. Define \(X_0^\ell(t)\), \(\ell=1,\,2,\,\dots,\,M\) as the ordered random walkers with the jump rates \eqref{eq:ellthr} and \eqref{eq:ellthl}. The sub-index 0 is a reminder that we have a boundary at site 0 only. Introduce the distances
\begin{equation}
 \begin{aligned}
  D_0^1(t):&=X_0^1(t),\\
  D_0^\ell(t):&=X_0^\ell(t)-X_0^{\ell-1}(t),\qquad1<\ell\le M.
 \end{aligned}\label{eq:disdef}
\end{equation}
\begin{pr}\label{pr:wdbl}
 The product measure (not necessarily \emph{probability} measure) \(\pi=\otimes_{\ell=1}^M\pi_\ell\) with marginals for \(z\ge1\)
 \begin{equation}
  \pi^\ell(z)=\Bigl(\prod_{n=\ell}^M\frac{P^n}{Q^n}\Bigr)^{z-1},\qquad1\le\ell\le M\label{eq:aseprev}
 \end{equation}
 is reversible stationary for the dynamics of the \(D_0\)'s.
\end{pr}
The structure of \(\pi^\ell\) reveals something remarkable which we cover later in Proposition \ref{pr:rh}. We need the \emph{Rankine-Hugoniot velocity}, which for one dimensional asymmetric particle systems is calculated as the difference of the hydrodynamic flux over that of the densities on the two sides of the macroscopic shock, and the flux is just the expected signed jump rate over an edge in the stationary model. For the parabola flux \((p-q)\vr(1-\vr)\) of ASEP this velocity takes the form
\begin{equation}
 (p-q)\cdot(1-\vr_0-\vr_K)
 =(p-q)\cdot(1-\vr_0-\vr_{K+1})\label{eq:aseprh}.
\end{equation}
Before we continue the stationary explorations for ASEP, we turn to BLP since the rest of the arguments will equally apply for both models.

\subsection{BLP in half line}

Similar ideas as in the ASEP case also reveal stationary distributions for BLP. We again start with the half-infinite setup with left boundary at 0. Recall \eqref{eq:rnk}, and notice that this time multiple second class particles can occupy any given site. We still denote their non-strictly ordered locations by \(X_0^\ell(t)\), \(\ell=1,\,2,\,\dots,\,M\) but this time a second class particle can only move left (right) if it has the lowest (highest, respectively) label on its site. With the same formal definitions \eqref{eq:disdef} of inter-walker distances as before, we now have \(D^1_0(t)\ge1\) but \(D^\ell_0(t)\ge0\) for \(1<\ell\le M\). The locations \(X^\ell_0(t)\) and therefore the number \(M_i(t)\ge0\) of second class particles at site \(i\) are a function of these:
\begin{equation}
 M_i(t)=\max\{\ell\,:\,X^\ell_0(t)\le i\}-\max\{\ell\,:\,X^\ell_0(t)\le i-1\},\qquad i=1\dots K\label{eq:midef}
\end{equation}
where \(\max\{\emptyset\}\) is understood as 0. We nevertheless use \(M_i(t)\) where convenient. With this notation (and omitting the time dependence), the right jump rate \eqref{eq:blpshrj} of the \(\ell^\text{th}\) second class particle is
\[
 P^\ell=\left\{
  \begin{aligned}
   &0,&&\text{if }\ell<M\text{ and }X^\ell_0=X^{\ell+1}_0,\\
   &\e{\te_0-\be(\ell-M_{X^\ell_0})}-\e{\te_0-\be\ell}
   =\e{\te_0-\be\ell}\Bigl(\e{\be M_{X^\ell_0}}-1\Bigr),&&\text{otherwise,}
  \end{aligned}
 \right.
\]
and the left jump rate \eqref{eq:blpshrb} is
\[
 Q^\ell=\left\{
  \begin{aligned}
   &0,&&\text{if }\ell>1\text{ and }X^{\ell-1}_0=X^\ell_0,\\
   &\e{-\te_0+\be(\ell+M_{X^\ell_0}-1)}-\e{-\te_0+\be(\ell-1)}
   =\e{-\te_0+\be(\ell-1)}\Bigl(\e{\be M_{X^\ell_0}}-1\Bigr),&&\text{otherwise.}
  \end{aligned}
 \right.
\]
After some trial and error one arrives at
\begin{pr}\label{pr:blprev}
 The (not necessarily probability) measure
 \begin{equation}
  \pi(d_0^1,\,\dots,\,d_0^M):\,=\Bigl(\bigotimes_{\ell=1}^Me^{\bigl(\be(\ell^2-2\ell+1-M^2)+2\te_0\cdot(M-\ell+1)\bigr)d^\ell_0}\Bigr)\cdot\Bigl(\bigotimes_{i=1}^K\prod_{n=1}^{m_i}\frac1{\e{\be n}-1}\Bigr),\label{eq:blpstati}
 \end{equation}
 where the \(m_i\) are calculated from the \(d_0^\ell\) as in \eqref{eq:midef} and empty products are defined to be 1, is reversible stationary for the \(D_0^\ell(t)\)'s.
\end{pr}

To derive the Rankine-Hunogiot velocity for BLP, the flux i.e., the expected jump rate, is calculated by \eqref{eq:ef}. Under \(\mu^\te\) it takes the form \(\e\te+\e{-\te}\), hence the Rankine-Hugoniot velocity this time is
\begin{equation}
 \frac{\e{\te_0-\be M}+\e{-\te_0+\be M}-\e\te_0-\e{-\te_0}}{-M}=\bigl(\e{2\te_0-\be M}-1\bigr)\cdot\e{-\te_0}\cdot\frac{\e{\be M}-1}M.\label{eq:blprh}
\end{equation}

\subsection{Stationarity}

In this section we conclude some properties that equally hold for ASEP and BLP.
\begin{pr}\label{pr:rh}
 The product measure \(\pi\) \eqref{eq:aseprev} and \eqref{eq:blpstati} for the half-line model with left boundary above can be normalised to a probability distribution if and only if the Rankine-Hugoniot velocity \eqref{eq:aseprh} or \eqref{eq:blprh}, respectively, is negative.
\end{pr}
This makes perfect sense as otherwise we expect the second class particles to diffuse or drift away from the left boundary. Their group velocity was shown to agree with this Rankine-Hugoniot formula in the no-boundary case (\cite[Appendix]{rwshscp} for ASEP, \cite[Lemmas 4.2, 4.3]{sokvalak} for BLP). Since they move as a tight group in the no-boundary case, one also expects that the normalisation first fails for \(\ell=1\) and the proofs of this proposition in Section \ref{sc:sd} indeed confirm this.

The half-line case with right boundary only is completely analogous.

Finally, the finite \(\Kc\) case with both boundaries present is easily obtained since restricting a reversible stationary Markov chain to a subset of its sample space only conditions the reversible stationarity. We denote the same \(M\) many \(D\) variables as \(D_{0K}\) to emphasise that both boundaries are present. The walkers \(X_{0K}\) are ordered as before, and they can be anywhere in \(\Kc\). Hence the \(D_{0K}\) are restricted to the space
\[
 \sum_{\ell=1}^MD_{0K}^\ell\le K.
\]
\begin{cor}
 For both ASEP and BLP, the reversible stationary distribution of the \(D_{0K}\)'s is given by
 \[
  \frac{\pi(\cdot)}{\pi\bigl\{\sum_{\ell=1}^MD_{0K}^\ell\le K\bigr\}}
 \]
 from the respective measures \eqref{eq:aseprev} and \eqref{eq:blpstati}.
\end{cor}
\begin{proof}
 Detailed balance on the allowed transitions works exactly as before.
\end{proof}

Finally, the cases when \(\pi\) can be normalised (either because the Rankine-Hugoniot velocity points towards the one-sided boundary, or in the finite volume case) allows a description of a stationary distribution of the models, this is a straightforward corollary of all our results so far.
\begin{cor}
 Fix \(M>0\) and assume \(\pi\) can be normalised to a probability distribution. Draw variables \(D^1,\,\dots D^M\) from this distribution, and define \(M_i\) as in \eqref{eq:midef}. Given these random occupations, write \(\un\nu^{\un\vr,\un m}\) from \eqref{eq:nurm}. The annealed outcome is stationary for the ASEP or BLP, respectively.
\end{cor}

Notice that our finite models when both boundaries are present have a fixed number of second class particles but are irreducible on this countable state space. This implies
\begin{cor}
 When both boundaries are present, the stationary distribution described above is unique on the irreducible state space of a given \(M\) number of second class particles.
\end{cor}
We leave the uniqueness question open for half-line models.

\section{Reverse engineering}

The proof of Theorem \ref{tm:mush} consists of revisiting the arguments for the bulk from \cite{rwshscp}, and brute-force finding the boundary rates \eqref{eq:asepbry} \eqref{eq:blpbry} that keep this construction valid in our finite volume scenario. In fact we can make direct use of the calculations in \cite{rwshscp} via the following trick. Recall \eqref{eq:rnk} and introduce the auxiliary random variables \(\om_0=\ze_0\sim\mu^{\vr_0}\) and \(\om_{K+1}=\ze_{K+1}\sim\mu^{\vr_{K+1}}\), mutually independent of each other and the rest of \((\un\om,\,\un\ze)\). Adding these auxiliary random variables, we can think of an extended pair \((\un\om,\,\un\ze)\) distributed according to the extended product measure \(\un\nu^{\un\vr,\un m}\) with no second class particles added to sites \(0\) and \(K+1\). The choice of no second class particles on the new sites is arbitrary but irrelevant, and keeps our argument simple.

Introduce the generator \(L^{(0\dots K+1)}\) by the same formula as \(\Lk\) of \eqref{eq:lbulkdef}, except that we run the summation from \(k=0\) to \(K\) and include the new variables \(\om_0=\ze_0\), \(\om_{K+1}=\ze_{K+1}\). We then have
\begin{equation}
 (\Lk\vp)(\un\om,\,\un\ze)=(L^{(0\dots K+1)}\vp)(\un\om,\,\un\ze)-\La\vp(\un\om,\,\un\ze),\label{eq:lb0a}
\end{equation}
with the bits concerning the auxiliary variables at sites 0 and \(K+1\) subtracted off:
\begin{equation}
 \begin{aligned}
  (\La\vp)(\un\om,\,\un\ze)=\phantom{[}&p(\om_0,\,\ze_1)\cdot\bigl[\vp(\un\om^{0,1},\,\un\ze^{0,1})-\vp(\un\om,\,\un\ze)\bigr]\\
  +\bigl[&p(\om_0,\,\om_1)-p(\om_0,\,\ze_1)\bigr]\cdot\bigl[\vp(\un\om^{0,1},\,\un\ze)-\vp(\un\om,\,\un\ze)\bigr]\\
  +\phantom{\bigl[}&q(\ze_0,\,\om_1)\cdot\bigl[\vp(\un\om^{1,0},\,\un\ze^{1,0})-\vp(\un\om,\,\un\ze)\bigr]\\
  +\bigl[&q(\ze_0,\,\ze_1)-q(\ze_0,\,\om_1)\bigr]\cdot\bigl[\vp(\un\om,\,\un\ze^{1,0})-\vp(\un\om,\,\un\ze)\bigr]\\
  +\phantom{\bigl[}&p(\om_K,\,\ze_{K+1})\cdot\bigl[\vp(\un\om^{K,K+1},\,\un\ze^{K,K+1})-\vp(\un\om,\,\un\ze)\bigr]\\
  +\bigl[&p(\ze_K,\,\ze_{K+1})-p(\om_K,\,\ze_{K+1})\bigr]\cdot\bigl[\vp(\un\om,\,\un\ze^{K,K+1})-\vp(\un\om,\,\un\ze)\bigr]\\
  +\phantom{\bigl[}&q(\ze_K,\,\om_{K+1})\cdot\bigl[\vp(\un\om^{K+1,K},\,\un\ze^{K+1,K})-\vp(\un\om,\,\un\ze)\bigr]\\
  +\bigl[&q(\om_K,\,\om_{K+1})-q(\ze_K,\,\om_{K+1})\bigr]\cdot\bigl[\vp(\un\om^{K+1,K},\,\un\ze)-\vp(\un\om,\,\un\ze)\bigr].
 \end{aligned}\label{eq:lauxdef}
\end{equation}
Introduce also the generator \(L^\Zb\) that has the summation of \eqref{eq:lbulkdef} on the entire \(\Zb\) (with further \((\om_k,\,\ze_k)\) variables).

\begin{rem}
 For the generator \(\Lk\) on the cylinder function \(\vp\,:\,I^\Kc\to\Rb\) with \(\Kc=\{1\dots K\}\), it doesn't matter what the hypothetical \((\om_k,\,\ze_k)\) variables are for \(k\) outside \(\Kc\) as nothing in \(\Lk\vp\) depends on these.
\end{rem}
 Hence we first extend the density to sites 0 and \(K+1\) via \eqref{eq:asepsokfelt} for ASEP and \eqref{eq:blpsokfelt} for BLP. We then proceed by extending the product measure \(\un\nu^{\un\vr,\un m}\) to the infinite volume \(\Zb\) by postulating that second class particles are only present in \(\Kc\), and \(\om_i=\ze_i\sim\mu^{\vr_0}\) for \(i\le0\); \(\om_i=\ze_i\sim\mu^{\vr_{K+1}}\) for \(i\ge K+1\), keeping full independence across sites of \(\Zb\). From now on, \(\Eb^{\un\vr,\un m}\) refers to this extended measure wherever we need it for a generator acting outside \(\Kc\).
With all these generators, our strategy is to use \cite{rwshscp} as much as possible. Hence we calculate the mean of \(\Lk\) in multiple steps:
\begin{enumerate}
 \item \eqref{eq:lb0a} allows us to calculate the mean of \(L^{(0\dots K+1)}\vp\) and \(\La\vp\) instead of that of \(\Lk\vp\).
 \item As \(\vp\) is cylinder on \(\Kc=\{1\dots K\}\), \(L^{(0\dots K+1)}\vp=L^\Zb\vp\) so we can turn to this latter instead.
 \item The calculation for the mean of \(L^\Zb\vp\) can be borrowed from \cite{rwshscp}.
 \item Hence we just need to calculate \(\La\vp\) under the expectation.
\end{enumerate}
This program is detailed by the lemmas below.
\begin{lm}\label{lm:zaux}
 For our original test function \(\vp\,:\,I^\Kc\to\Rb\) and any distribution (the expectation of which we denote by \(\Eb\)),
 \[
  \Eb(\Lk\vp)(\un\om,\,\un\ze)=
  \Eb(L^\Zb\vp)(\un\om,\,\un\ze)-
  \Eb(\La\vp)(\un\om,\,\un\ze).
 \]
\end{lm}
\begin{proof}
 Terms with jumps outside the range \(\Kc\) of dependence for \(\vp\) do not play a role, hence \(L^{(0\dots K+1)}\vp=L^\Zb\vp\). The proof now follows from \eqref{eq:lb0a}.
\end{proof}
From now on we omit the state dependence of \(\vp(\un\om,\,\un\ze)\) and simply write \(\vp\) in case the argument is just \((\un\om,\,\un\ze)\). Further to \eqref{eq:jump}, we also extend the notation \eqref{eq:1sp} to the second class particle configuration \(\un m\).
\begin{lm}\label{lm:lzcomp}
 \[
  \begin{aligned}
   \Eb^{\un\vr,\un m}(L^\Zb\vp)&=\sum_{i=1}^{K-1}P(m_i,\,m_{i+1},\,\vr_i,\,\vr_{i+1})\cdot\bigl[\Eb^{\vrir,\un m^{i,i+1}}\vp-\Eb^{\un\vr,\un m}\vp\bigr]\\
   &\qquad+P(m_K,\,0,\,\vr_K,\,\vr_K)\cdot\bigl[\Eb^{\un\vr^{K\rightarrow},\un m^{K-}}\vp-\Eb^{\un\vr,\un m}\vp\bigr]\\
   &+\sum_{i=1}^{K-1}Q(m_i,\,m_{i+1},\,\vr_i,\,\vr_{i+1})\cdot\bigl[\Eb^{\vril,\un m^{i+1,i}}\vp-\Eb^{\un\vr,\un m}\vp\bigr]\\
   &\qquad+Q(0,\,m_1,\,\vr_0,\,\vr_1)\cdot\bigl[\Eb^{\un\vr,\un m^{1-}}\vp-\Eb^{\un\vr,\un m}\vp\bigr]
  \end{aligned}
 \]
 with densities \(\un\vr\) and random walk jump rates \(Q\) and \(P\) as in \eqref{eq:asepsokfelt}-\eqref{eq:asepshrb}, for ASEP and \eqref{eq:blpsokfelt}-\eqref{eq:blpshrb} for BLP.
\end{lm}
\begin{proof}
 We are almost exactly in the situation of \cite[Theorem 2]{rwshscp}. Its proof starts with the assumption that there are no second class particles on the leftmost and rightmost sites of dependence for the cylinder function. This does not necessarily hold in our case. However, \(\vp\) is also a cylinder function on the sites \(0,\,1,\,\dots,\,K+1\), in which case this assumption becomes valid. Hence the full proof applies and results in
 \[
  \begin{aligned}
   \Eb^{\un\vr,\un m}(L^\Zb\vp)&=\sum_{i=0}^KP(m_i,\,m_{i+1},\,\vr_i,\,\vr_{i+1})\cdot\bigl[\Eb^{\vrir,\un m^{i,i+1}}\vp-\Eb^{\un\vr,\un m}\vp\bigr]\\
   &+\sum_{i=0}^KQ(m_i,\,m_{i+1},\,\vr_i,\,\vr_{i+1})\cdot\bigl[\Eb^{\vril,\un m^{i+1,i}}\vp-\Eb^{\un\vr,\un m}\vp\bigr].
  \end{aligned}
 \]
 The final step is a careful look at the boundaries of the summations. Notice that for both ASEP and BLP, \(P(0,\,\cdot,\,\cdot,\,\cdot)=Q(\cdot,\,0,\,\cdot,\,\cdot)=0\), and a.s.\ there are no second class particles on sites 0 and \(K+1\): \(m_0=m_{K+1}=0\). This takes care of the \(i=0\) term in the first line and the \(i=K\) term in the second one. The opposite boundaries do see some action, there we will make use of the product structure of the distribution and the fact that \(\vp\) does not depend on \((\om_0,\,\ze_0)\), nor \((\om_{K+1},\,\ze_{K+1})\). Hence, with notation similar to \eqref{eq:1sp},
 \[
  \Eb^{\un\vr^{0\leftarrow},\un m^{1,0}}\vp=
  \Eb^{\un\vr,\un m^{1-}}\vp
  \qquad\text{and}\qquad
  \Eb^{\un\vr^{K\rightarrow},\un m^{K,K+1}}\vp=
  \Eb^{\un\vr^{K\rightarrow},\un m^{K-}}\vp
 \]
 which, with \(m_0=m_{K+1}=0\) and therefore \(\vr_{K+1}=\vr_K\), finishes the proof.
\end{proof}
\begin{lm}\label{lm:lacomp}
 \begin{multline}
  \Eb^{\un\vr,\un m}(\La\vp)(\un\om,\,\un\ze)\\
  \begin{aligned}
   =\Eb^{\un\vr,\un m}\bigl\{&p(\om_0,\,\ze_1-1)\cdot\frac{\hat\mu^{\vr_1,m_1}(\om_1-1)}{\hat\mu^{\vr_1,m_1}(\om_1)}-p(\om_0,\,\om_1)+q(\ze_0,\,\om_1+1)\cdot\frac{\hat\mu^{\vr_1,m_1}(\om_1+1)}{\hat\mu^{\vr_1,m_1}(\om_1)}-q(\ze_0,\,\ze_1)\\
   +&p(\om_K+1,\,\ze_{K+1})\cdot\frac{\hat\mu^{\vr_K,m_K}(\om_K+1)}{\hat\mu^{\vr_K,m_K}(\om_K)}-p(\ze_K,\,\ze_{K+1})\\
   +&q(\ze_K-1,\,\om_{K+1})\cdot\frac{\hat\mu^{\vr_K,m_K}(\om_K-1)}{\hat\mu^{\vr_K,m_K}(\om_K)}-q(\om_K,\,\om_{K+1})\bigr\}\cdot\vp\\
   +\Eb^{\un\vr,\un m^{1-}}\bigl\{\bigl[&p(\om_0,\,\om_1-1)-p(\om_0,\,\ze_1)\bigr]\cdot\frac{\hat\mu^{\vr_1,m_1}(\om_1-1)}{\hat\mu^{\vr_1,m_1-1}(\om_1)}\\
   +\bigl[&q(\ze_0,\,\ze_1+1)-q(\ze_0,\,\om_1)\bigr]\cdot\frac{\hat\mu^{\vr_1,m_1}(\om_1)}{\hat\mu^{\vr_1,m_1-1}(\om_1)}\bigr\}\cdot\vp\\
   +\Eb^{\un\vr,\un m^{K-}}\bigl\{\bigl[&p(\ze_K+1,\,\ze_{K+1})-p(\om_K,\,\ze_{K+1})\bigr]\cdot\frac{\hat\mu^{\vr_K,m_K}(\om_K)}{\hat\mu^{\vr_K,m_K-1}(\om_K)}\\
   +\bigl[&q(\om_K-1,\,\om_{K+1})-q(\ze_K,\,\om_{K+1})\bigr]\cdot\frac{\hat\mu^{\vr_K,m_K}(\om_K-1)}{\hat\mu^{\vr_K,m_K-1}(\om_K)}\bigr\}\cdot\vp.
  \end{aligned}\label{eq:lauxexp}
 \end{multline}
 Here the last two+two lines are understood to be zero if \(m_1=0\) or \(m_K=0\), respectively, and each line under the expectations is meant to be zero if any of the \(\om_i\) or \(\ze_i\) variables is out of the range \(I\).
\end{lm}
\begin{proof}
 When calculating \(\Eb^{\un\vr,\un m}(\La\vp)\), we will change summation variables \(\om_1\), \(\ze_1\), \(\om_K\) and \(\ze_K\) of \(\Eb^{\un\vr,\un m}\) as needed to undo the changes inside \(\vp\) of \eqref{eq:lauxdef}. This will also result in Radon-Nikodym terms with the measures \(\hat\mu^{\vr_i,m_i}\) and, when the number of second class particles also changes, we will need to move the vector \(\un m\) to avoid singularities. As \(\vp\) is cylinder on \(\Kc\), we do not change the variables on sites \(0\) and \(K+1\). Hence, the new summation variables for the first \(\vp\) in the respective lines of \eqref{eq:lauxdef} are as follows:
 \begin{center}
  \begin{tabular}{r|c|c|}
   Line&Old&New\\
   \hline
   1.&\((\un\om^{1+},\,\un\ze^{1+})\)&\(=(\un\om,\,\un\ze)\)\\
   \hline
   2.&\((\un\om^{1+},\,\un\ze)\)&\(=(\un\om,\,\un\ze)\)\\
   \hline
   3.&\((\un\om^{1-},\,\un\ze^{1-})\)&\(=(\un\om,\,\un\ze)\)\\
   \hline
   4.&\((\un\om,\,\un\ze^{1-})\)&\(=(\un\om,\,\un\ze)\)\\
   \hline
   5.&\((\un\om^{K-},\,\un\ze^{K-})\)&\(=(\un\om,\,\un\ze)\)\\
   \hline
   6.&\((\un\om,\,\un\ze^{K-})\)&\(=(\un\om,\,\un\ze)\)\\
   \hline
   7.&\((\un\om^{K+},\,\un\ze^{K+})\)&\(=(\un\om,\,\un\ze)\)\\
   \hline
   8.&\((\un\om^{K+},\,\un\ze)\)&\(=(\un\om,\,\un\ze)\)\\
   \hline
  \end{tabular}
 \end{center}
 The new variables in lines 2, 4, 6 and 8 have modified numbers of second class particles, hence these states are singular to the original expectation \(\Eb^{\un\vr,\un m}\). Instead, we need to use \(\Eb^{\un\vr,\un m^{1-}}\), \(\Eb^{\un\vr,\un m^{1-}}\), \(\Eb^{\un\vr,\un m^{K-}}\), \(\Eb^{\un\vr,\un m^{K-}}\) respectively, with the notation \eqref{eq:1sp} extended for \(\un m\) as well. We do not change the variables for the second terms \(\vp\) in each line.

 Extra care is required in the ASEP case, as the changes above in some cases lead out of the state space of zero or one particles per site. A careful inspection reveals that the below makes sense both for ASEP and BLP if each line is understood with the indicator that all variables in the respective Radon-Nikodym terms and inside the \(\vp\)'s are within range. For example, both lines 9 and 10 below are understood as zero if \(m_1=0\). Further, for ASEP, line 9 is also zero for \(\om_1=0\) as \(\om_1^{1-}\) would then go negative, while line 10 is zero when \(\om_1=1\) as the definition of \(\hat\mu\) included \(\hat\mu^{\cdot,1}(1)=0\) in the ASEP case. These statements can be traced down from before the change of variables. We apply this convention to each Radon-Nikodym term for the rest of the paper.

 With these changes we arrive to
 \[
  \begin{aligned}
   \Eb^{\un\vr,\un m}(\La\vp)(\un\om,\,\un\ze)=\Eb^{\un\vr,\un m}\bigl\{&p(\om_0,\,\ze_1^{1-})\cdot\frac{\hat\mu^{\vr_1,m_1}(\om_1^{1-})}{\hat\mu^{\vr_1,m_1}(\om_1)}-p(\om_0,\,\ze_1)\\
    -\bigl[&p(\om_0,\,\om_1)-p(\om_0,\,\ze_1)\bigr]\\
    +&q(\ze_0,\,\om_1^{1+})\cdot\frac{\hat\mu^{\vr_1,m_1}(\om_1^{1+})}{\hat\mu^{\vr_1,m_1}(\om_1)}-q(\ze_0,\,\om_1)\\
    -\bigl[&q(\ze_0,\,\ze_1)-q(\ze_0,\,\om_1)\bigr]\\
    +&p(\om_K^{K+},\,\ze_{K+1})\cdot\frac{\hat\mu^{\vr_K,m_K}(\om_K^{K+})}{\hat\mu^{\vr_K,m_K}(\om_K)}-p(\om_K,\,\ze_{K+1})\\
    -\bigl[&p(\ze_K,\,\ze_{K+1})-p(\om_K,\,\ze_{K+1})\bigr]\\
    +&q(\ze_K^{K-},\,\om_{K+1})\cdot\frac{\hat\mu^{\vr_K,m_K}(\om_K^{K-})}{\hat\mu^{\vr_K,m_K}(\om_K)}-q(\ze_K,\,\om_{K+1})\\
    -\bigl[&q(\om_K,\,\om_{K+1})-q(\ze_K,\,\om_{K+1})\bigr]\bigr\}\cdot\vp\\
    +\Eb^{\un\vr,\un m^{1-}}\bigl\{\bigl[&p(\om_0,\,\om_1^{1-})-p(\om_0,\,\ze_1)\bigr]\cdot\frac{\hat\mu^{\vr_1,m_1}(\om_1^{1-})}{\hat\mu^{\vr_1,m_1-1}(\om_1)}\\
    +\bigl[&q(\ze_0,\,\ze_1^{1+})-q(\ze_0,\,\om_1)\bigr]\cdot\frac{\hat\mu^{\vr_1,m_1}(\om_1)}{\hat\mu^{\vr_1,m_1-1}(\om_1)}\bigr\}\cdot\vp\\
    +\Eb^{\un\vr,\un m^{K-}}\bigl\{\bigl[&p(\ze_K^{K+},\,\ze_{K+1})-p(\om_K,\,\ze_{K+1})\bigr]\cdot\frac{\hat\mu^{\vr_K,m_K}(\om_K)}{\hat\mu^{\vr_K,m_K-1}(\om_K)}\\
    +\bigl[&q(\om_K^{K-},\,\om_{K+1})-q(\ze_K,\,\om_{K+1})\bigr]\cdot\frac{\hat\mu^{\vr_K,m_K}(\om_K^{K-})}{\hat\mu^{\vr_K,m_K-1}(\om_K)}\bigr\}\cdot\vp.
  \end{aligned}
 \]
 Cancelling a few terms gives \eqref{eq:lauxexp}.
\end{proof}
Next we perform a similar change of variables on the four lines of \eqref{eq:lbdydef}. As no second class particles move, there is no issue with changing the vector \(\un m\). Again, we do not change anything for the second parts \(\vp\).
\begin{lm}\label{lm:lycomp}
 \[
  \begin{aligned}
  \Eb^{\un\vr,\un m}(\Ly\vp)(\un\om,\,\un\ze)=\Eb^{\un\vr,\un m}\bigl\{&p_0(\om_1-1,\,\ze_1-1)\cdot\frac{\hat\mu^{\vr_1,m_1}(\om_1-1)}{\hat\mu^{\vr_1,m_1}(\om_1)}-p_0(\om_1,\,\ze_1)\\
   +&q_0(\om_1+1,\,\ze_1+1)\cdot\frac{\hat\mu^{\vr_1,m_1}(\om_1+1)}{\hat\mu^{\vr_1,m_1}(\om_1)}-q_0(\om_1,\,\ze_1)\\
   +&p_K(\om_K+1,\,\ze_K+1)\cdot\frac{\hat\mu^{\vr_K,m_K}(\om_K+1)}{\hat\mu^{\vr_K,m_K}(\om_K)}-p_K(\om_K,\,\ze_K)\\
   +&q_K(\om_K-1,\,\ze_K-1)\cdot\frac{\hat\mu^{\vr_K,m_K}(\om_K-1)}{\hat\mu^{\vr_K,m_K}(\om_K)}-q_K(\om_K,\,\ze_K)\bigr\}\cdot\vp.
  \end{aligned}
 \]
\end{lm}
\begin{proof}
 As before, the new summation variables for the first \(\vp\) in the respective lines of \eqref{eq:lbdydef} are as follows:
 \begin{center}
  \begin{tabular}{r|c|c|}
   Line&Old&New\\
   \hline
   1.&\((\un\om^{1+},\,\un\ze^{1+})\)&\(=(\un\om,\,\un\ze)\)\\
   \hline
   2.&\((\un\om^{1-},\,\un\ze^{1-})\)&\(=(\un\om,\,\un\ze)\)\\
   \hline
   3.&\((\un\om^{K-},\,\un\ze^{K-})\)&\(=(\un\om,\,\un\ze)\)\\
   \hline
   4.&\((\un\om^{K+},\,\un\ze^{K+})\)&\(=(\un\om,\,\un\ze)\)\\
   \hline
  \end{tabular}
 \end{center}
 The statement follows after factoring in the Radon-Nikodym derivatives.
\end{proof}

Let us summarise how far we got so far. We are in the process of calculating, for a cylinder function \(\vp\) depending on \(\omega\)'s of \(\Kc\),
\[
 \begin{aligned}
  \Eb^{\un\vr,\un m}(L\vp)(\un\om,\,\un\ze)&=\Eb^{\un\vr,\un m}(\Lk\vp)(\un\om,\,\un\ze)+\Eb^{\un\vr,\un m}(\Ly\vp)(\un\om,\,\un\ze)\\
  &=\Eb^{\un\vr,\un m}(L^\Zb\vp)(\un\om,\,\un\ze)-\Eb^{\un\vr,\un m}(\La\vp)(\un\om,\,\un\ze)+\Eb^{\un\vr,\un m}(\Ly\vp)(\un\om,\,\un\ze).
 \end{aligned}
\]
Here we have so far used the decomposition \eqref{eq:ldef} of the generator of our process, then Lemma \ref{lm:zaux}.
\begin{proof}[Proof of Theorem \ref{tm:mush}]
Combine with Lemmas \ref{lm:lzcomp}, \ref{lm:lacomp} and \ref{lm:lycomp} to get 
\begin{multline*}
 \Eb^{\un\vr,\un m}(L\vp)(\un\om,\,\un\ze)\\
 \begin{aligned}
&=\sum_{i=1}^{K-1}P(m_i,\,m_{i+1},\,\vr_i,\,\vr_{i+1})\cdot\bigl[\Eb^{\vrir,\un m^{i,i+1}}\vp-\Eb^{\un\vr,\un m}\vp\bigr]\\
   &\qquad+P(m_K,\,0,\,\vr_K,\,\vr_K)\cdot\bigl[\Eb^{\un\vr^{K\rightarrow},\un m^{K-}}\vp-\Eb^{\un\vr,\un m}\vp\bigr]\\
   &+\sum_{i=1}^{K-1}Q(m_i,\,m_{i+1},\,\vr_i,\,\vr_{i+1})\cdot\bigl[\Eb^{\vril,\un m^{i+1,i}}\vp-\Eb^{\un\vr,\un m}\vp\bigr]\\
   &\qquad+Q(0,\,m_1,\,\vr_0,\,\vr_1)\cdot\bigl[\Eb^{\un\vr,\un m^{1-}}\vp-\Eb^{\un\vr,\un m}\vp\bigr]\\
   &-\Eb^{\un\vr,\un m}\bigl\{p(\om_0,\,\ze_1-1)\cdot\frac{\hat\mu^{\vr_1,m_1}(\om_1-1)}{\hat\mu^{\vr_1,m_1}(\om_1)}-p(\om_0,\,\om_1)+q(\ze_0,\,\om_1+1)\cdot\frac{\hat\mu^{\vr_1,m_1}(\om_1+1)}{\hat\mu^{\vr_1,m_1}(\om_1)}-q(\ze_0,\,\ze_1)\\
   &\qquad+p(\om_K+1,\,\ze_{K+1})\cdot\frac{\hat\mu^{\vr_K,m_K}(\om_K+1)}{\hat\mu^{\vr_K,m_K}(\om_K)}-p(\ze_K,\,\ze_{K+1})\\
   &\qquad+q(\ze_K-1,\,\om_{K+1})\cdot\frac{\hat\mu^{\vr_K,m_K}(\om_K-1)}{\hat\mu^{\vr_K,m_K}(\om_K)}-q(\om_K,\,\om_{K+1})\bigr\}\cdot\vp\\
   &-\Eb^{\un\vr,\un m^{1-}}\bigl\{\bigl[p(\om_0,\,\om_1-1)-p(\om_0,\,\ze_1)\bigr]\cdot\frac{\hat\mu^{\vr_1,m_1}(\om_1-1)}{\hat\mu^{\vr_1,m_1-1}(\om_1)}\\
   &\qquad+\bigl[q(\ze_0,\,\ze_1+1)-q(\ze_0,\,\om_1)\bigr]\cdot\frac{\hat\mu^{\vr_1,m_1}(\om_1)}{\hat\mu^{\vr_1,m_1-1}(\om_1)}\bigr\}\cdot\vp\\
   &-\Eb^{\un\vr,\un m^{K-}}\bigl\{\bigl[p(\ze_K+1,\,\ze_{K+1})-p(\om_K,\,\ze_{K+1})\bigr]\cdot\frac{\hat\mu^{\vr_K,m_K}(\om_K)}{\hat\mu^{\vr_K,m_K-1}(\om_K)}\\
   &\qquad+\bigl[q(\om_K-1,\,\om_{K+1})-q(\ze_K,\,\om_{K+1})\bigr]\cdot\frac{\hat\mu^{\vr_K,m_K}(\om_K-1)}{\hat\mu^{\vr_K,m_K-1}(\om_K)}\bigr\}\cdot\vp\\
   &+\Eb^{\un\vr,\un m}\bigl\{p_0(\om_1-1,\,\ze_1-1)\cdot\frac{\hat\mu^{\vr_1,m_1}(\om_1-1)}{\hat\mu^{\vr_1,m_1}(\om_1)}-p_0(\om_1,\,\ze_1)\\
   &\qquad+q_0(\om_1+1,\,\ze_1+1)\cdot\frac{\hat\mu^{\vr_1,m_1}(\om_1+1)}{\hat\mu^{\vr_1,m_1}(\om_1)}-q_0(\om_1,\,\ze_1)\\
   &\qquad+p_K(\om_K+1,\,\ze_K+1)\cdot\frac{\hat\mu^{\vr_K,m_K}(\om_K+1)}{\hat\mu^{\vr_K,m_K}(\om_K)}-p_K(\om_K,\,\ze_K)\\
   &\qquad+q_K(\om_K-1,\,\ze_K-1)\cdot\frac{\hat\mu^{\vr_K,m_K}(\om_K-1)}{\hat\mu^{\vr_K,m_K}(\om_K)}-q_K(\om_K,\,\ze_K)\bigr\}\cdot\vp.
 \end{aligned}
\end{multline*}
Notice that the two sums already correspond to \eqref{eq:mush}. We need all other terms to cancel out. As \(\nu^{\cdot,\un m^{K-}}\), \(\nu^{\cdot,\un m^{1-}}\) and \(\nu^{\cdot,\un m}\) are mutually singular, the cancellation must happen within these respective expectations.

The arbitrary test function \(\vp\) forces identities for each \(\om_i,\,\ze_i\) for \(i\in\Kc\). However, it does not depend on variables at sites \(0\) and \(K+1\) hence, due to independence, here the expectation applies without the \(\vp\). Notice also \(0=m_0=\ze_0-\om_0\). With these in mind, the terms under \(\Eb^{\un\vr,\un m^{1-}}\) require
\begin{equation}
 \begin{aligned}
  &Q(0,\,m_1,\,\vr_0,\,\vr_1)\\
  &-\sum_{\om_0}\Bigl\{\bigl[p(\om_0,\,\om_1-1)-p(\om_0,\,\ze_1)\bigr]\cdot\frac{\hat\mu^{\vr_1,m_1}(\om_1-1)}{\hat\mu^{\vr_1,m_1-1}(\om_1)}\\
  &+\bigl[q(\om_0,\,\ze_1+1)-q(\om_0,\,\om_1)\bigr]\cdot\frac{\hat\mu^{\vr_1,m_1}(\om_1)}{\hat\mu^{\vr_1,m_1-1}(\om_1)}\Bigr\}\hat\mu^{\vr_0,0}(\om_0)=0
 \end{aligned}\label{eq:qleft}
\end{equation}
for any pair \(\un\ze=\un\om+\un m^{1-}\), which now translates to just \(m_1-1=\ze_1-\om_1\). Notice, as we remarked earlier, that only \(\om\) and \(\ze\) values in the state space \(I^\Zb\) are making a contribution. This is due to the factor \(\vp\) at the end of each line in the generator calculation above.

We now check this holds for ASEP. As noted before, the case \(m_1=0\) gives zero in both Radon-Nikodym derivatives. It also results in zero for the \(Q\) term by the definition \eqref{eq:asepshrb}.

With \(m_1=1\), we have \(\ze_1=\om_1\) and
\[
 \vr_1=\frac{p\vr_0}{p\vr_0+q(1-\vr_0)}
\]
from \eqref{eq:asepsokfelt}. When \(\om_1=0\), the Radon-Nikodym term behind the \(p\) terms is zero, and the sum in \eqref{eq:qleft} turns into (cf.\ \eqref{eq:asepsokmuhat})
\[
 \bigl[q(0,1)-q(0,0)\bigr]\cdot\frac1{1-\vr_1}(1-\vr_0)+
 \bigl[q(1,1)-q(1,0)\bigr]\cdot\frac1{1-\vr_1}\vr_0
 =q\cdot\frac{1-\vr_0}{1-\vr_1}=p\vr_0+q(1-\vr_0).
\]
We used the last display in the last step.

When \(\om_1=1\), \(m_1=1\) sends the argument of \(\hat\mu^{\vr_1,1}(1)\) out of range in the Radon-Nikodym term behind the \(q\)'s, and leaves us with the sum equaling
\[
 \bigl[p(0,0)-p(0,1)\bigr]\cdot\frac1{\vr_1}(1-\vr_0)+
 \bigl[p(1,0)-p(1,1)\bigr]\cdot\frac1{\vr_1}\vr_0
 =p\cdot\frac{\vr_0}{\vr_1}=p\vr_0+q(1-\vr_0).
\]
Both match with definition \eqref{eq:asepshrb} of \(Q(0,\,1,\,\vr_0,\,\vr_1)\).

Next we check \eqref{eq:qleft} in the BLP. The Radon-Nikodym term still requires \(m_1\ge1\), but this is the only restriction as now \(I=\Zb\). The sum now reads, via \eqref{eq:blppdef}, \eqref{eq:gauss} and \eqref{eq:blpsokmuhat},
\[
 \begin{aligned}
  &\sum_{\om_0=-\infty}^\infty\bigl[\e{\be(\frac12-\om_1)}-\e{\be(\frac12-\om_1-m_1)}\bigr]\cdot\e{-\frac\be2\cdot[(\om_1-1-\frac{\te(\vr_1)}\be)^2-(\om_1-\frac{\te(\vr_1)}\be)^2]}\hat\mu^{\vr_0,0}(\om_0)\\
  &=\e{\be(\frac12-\om_1)}\bigl[1-\e{-\be m_1}\bigr]\cdot\e{\be\cdot(\om_1-\frac12-\frac{\te(\vr_1)}\be)}=\e{-\te(\vr_1)}-\e{-\te(\vr_1)-\be m_1},
 \end{aligned}
\]
which again matches \(Q(0,\,m_1,\,\vr_0,\,\vr_1)\) of \eqref{eq:blpshrb} after recursive use of \eqref{eq:teshi}.

Next we turn to the \(\Eb^{\cdot,\un m^{K-}}\) terms. The one with the \(P\) factor in front has a modified density \(\vr^{K\rightarrow}_K\) at site \(K\) in the expectation. As the other terms come with an unmodified density, we start with correcting this using a Radon-Nikodym derivative:
\[
 \Eb^{\un\vr^{K\rightarrow},\un m^{K-}}\vp
 =\Eb^{\un\vr,\un m^{K-}}\Bigl(\frac{\hat\mu^{\vr^{K\rightarrow}_K,m_K-1}(\om_K)}{\hat\mu^{\vr_K,m_K-1}(\om_K)}\cdot\vp\Bigr).
\]
With this, and due to \(\vp\) being arbitrary cylinder on \(\Kc\), all terms under \(\Eb^{\cdot,\un m^{K-}}\) expectations must cancel out as:
\begin{equation}
 \begin{aligned}
  &P(m_K,\,0,\,\vr_K,\,\vr_K)\cdot\frac{\hat\mu^{\vr^{K\rightarrow}_K,m_K-1}(\om_K)}{\hat\mu^{\vr_K,m_K-1}(\om_K)}\\
  &-\sum_{\om_{K+1}}\Bigl\{\bigl[p(\ze_K+1,\,\om_{K+1})-p(\om_K,\,\om_{K+1})\bigr]\cdot\frac{\hat\mu^{\vr_K,m_K}(\om_K)}{\hat\mu^{\vr_K,m_K-1}(\om_K)}\\
  &+\bigl[q(\om_K-1,\,\om_{K+1})-q(\ze_K,\,\om_{K+1})\bigr]\cdot\frac{\hat\mu^{\vr_K,m_K}(\om_K-1)}{\hat\mu^{\vr_K,m_K-1}(\om_K)}\Bigr\}\hat\mu^{\vr_K,0}(\om_{K+1})=0.
 \end{aligned}\label{eq:pright}
\end{equation}
Here we already incorporated \(m_{K+1}=0\), hence \(\ze_{K+1}=\om_{K+1}\). This display is verified similarly as in the previous round. First, the case \(m_K=0\) makes all terms trivially zero.

For ASEP this leaves us with \(m_K=1\), hence \(\ze_K=\om_K+m_K-1=\om_K\) and
\[
 \vr^{K\rightarrow}_K=\frac{q\vr_K}{p-(p-q)\vr_K}
\]
by \eqref{eq:asepvrj}. When \(\ze_K=\om_K=0\), \eqref{eq:asepshrj} and \eqref{eq:asepsokmuhat} turn \eqref{eq:pright} into
\[
 \begin{aligned}
  &\quad P(1,\,0,\,\vr_K,\,\vr_K)\cdot\frac{1-\frac{q\vr_K}{p-(p-q)\vr_K}}{1-\vr_K}-p(1,\,0)\cdot\frac1{1-\vr_K}(1-\vr_K)\\
  &=P(1,\,0,\,\vr_K,\,\vr_K)\cdot\frac p{p-(p-q)\vr_K}-p=0,
 \end{aligned}
\]
while for \(\ze_K=\om_K=1\), we have
\[
 P(1,\,0,\,\vr_K,\,\vr_K)\cdot\frac{\frac{q\vr_K}{p-(p-q)\vr_K}}{\vr_K}-q(0,\,1)\cdot\frac1{\vr_K}\vr_K=P(1,\,0,\,\vr_K,\,\vr_K)\cdot\frac q{p-(p-q)\vr_K}-q=0,
\]
both matching definition \eqref{eq:asepshrj}.

For BLP, we are once again restricted to \(m_K\ge1\) and \(\ze_K=\om_K+m_K-1\). Expanding the last sum and using \eqref{eq:blppdef}, \eqref{eq:gauss} and \eqref{eq:blpsokmuhat},
\[
 \sum_{\om_{K+1}=-\infty}^\infty\bigl[\e{\be(\om_K+m_K-\frac12)}-\e{\be(\om_K-\frac12)}\bigr]\hat\mu^{\vr_K,0}(\om_{K+1})=\e{\be(\om_K+m_K-\frac12)}-\e{\be(\om_K-\frac12)},
\]
while the first term of \eqref{eq:pright}, using \eqref{eq:gauss} and \eqref{eq:rt},
\[
 \begin{aligned}
  P(m_K,\,0,\,\vr_K,\,\vr_K)\cdot\frac{\mu^{\vr_K+1}(\om_K)}{\mu^{\vr_K}(\om_K)}&=P(m_K,\,0,\,\vr_K,\,\vr_K)\cdot\frac{\e{\te(\vr_K+1)\om_K}Z\bigl(\te(\vr_K)\bigr)}{\e{\te(\vr_K)\om_K}Z\bigl(\te(\vr_K+1)\bigr)}\\
  &=P(m_K,\,0,\,\vr_K,\,\vr_K)\cdot\frac{\e{(\te(\vr_K)+\be)\om_K}Z\bigl(\te(\vr_K)\bigr)}{\e{\te(\vr_K)\om_K}Z\bigl(\te(\vr_K)+\be\bigr)}\\
  &=P(m_K,\,0,\,\vr_K,\,\vr_K)\cdot\e{\be\om_K}\cdot\e{\be/2-\bigl(\te(\vr_K)+\be\bigr)}\\
  &=\bigl(\e{\te(\vr_K+m_K)}-\e{\te(\vr_K)}\bigr)\cdot\e{\be\om_K-\te(\vr_K)-\be/2}\\
  &=\bigl(\e{\be m_K}-1\bigr)\cdot\e{\be\om_K-\be/2}
 \end{aligned}
\]
is the same i.e., \eqref{eq:pright} is verified.

We arrived to the actual reverse-engineering stage. This arises from the \(\Eb^{\un\vr,\un m}(\cdot\times\vp)\) parts, except those under the two sums which have already been accounted for, summing to zero. As \(\vp\) is cylinder on \(\Kc\), the expectation on sites 0 and \(K+1\) can be calculated, and we are left with functions of variables on sites 1 and \(K\). We group these two sides seprately, keeping in mind that they might exchange a constant between each other. We also substitue \(\ze_0=\om_0\), \(\ze_1=\om_1+m_1\), \(\ze_K=\om_K+m_K\), and \(\ze_{K+1}=\om_{K+1}\). As \(\vp\) depends on sites 1 and \(K\) in an arbitrary manner, identities must hold for each \(\om_1\) and \(\om_K\) in \(I\). Recall also \(\vr_K=\vr_{K+1}\). Hence, we have
\begin{equation}
 \begin{aligned}
   &-Q(0,\,m_1,\,\vr_0,\,\vr_1)\\
   &-\sum_{\om_0}\bigl\{p(\om_0,\,\om_1+m_1-1)\cdot\frac{\hat\mu^{\vr_1,m_1}(\om_1-1)}{\hat\mu^{\vr_1,m_1}(\om_1)}-p(\om_0,\,\om_1)\\
   &\qquad+q(\om_0,\,\om_1+1)\cdot\frac{\hat\mu^{\vr_1,m_1}(\om_1+1)}{\hat\mu^{\vr_1,m_1}(\om_1)}-q(\om_0,\,\om_1+m_1)\bigr\}\cdot\hat\mu^{\vr_0,0}(\om_0)\\
   &+p_0(\om_1-1,\,\om_1+m_1-1)\cdot\frac{\hat\mu^{\vr_1,m_1}(\om_1-1)}{\hat\mu^{\vr_1,m_1}(\om_1)}-p_0(\om_1,\,\om_1+m_1)\\
   &+q_0(\om_1+1,\,\om_1+m_1+1)\cdot\frac{\hat\mu^{\vr_1,m_1}(\om_1+1)}{\hat\mu^{\vr_1,m_1}(\om_1)}-q_0(\om_1,\,\om_1+m_1)\\
   &-P(m_K,\,0,\,\vr_K,\,\vr_K)\\
   &-\sum_{\om_{K+1}}\bigl\{p(\om_K+1,\,\om_{K+1})\cdot\frac{\hat\mu^{\vr_K,m_K}(\om_K+1)}{\hat\mu^{\vr_K,m_K}(\om_K)}-p(\om_K+m_K,\,\om_{K+1})\\
   &\qquad+q(\om_K+m_K-1,\,\om_{K+1})\cdot\frac{\hat\mu^{\vr_K,m_K}(\om_K-1)}{\hat\mu^{\vr_K,m_K}(\om_K)}-q(\om_K,\,\om_{K+1})\bigr\}\cdot\hat\mu^{\vr_K,0}(\om_{K+1})\\
   &+p_K(\om_K+1,\,\om_K+m_K+1)\cdot\frac{\hat\mu^{\vr_K,m_K}(\om_K+1)}{\hat\mu^{\vr_K,m_K}(\om_K)}-p_K(\om_K,\,\om_K+m_K)\\
   &+q_K(\om_K-1,\,\om_K+m_K-1)\cdot\frac{\hat\mu^{\vr_K,m_K}(\om_K-1)}{\hat\mu^{\vr_K,m_K}(\om_K)}-q_K(\om_K,\,\om_K+m_K)=0.
 \end{aligned}\label{eq:reveng}
\end{equation}
We handle the two halves, depending on \(\om_1\) and \(\om_K\) respectively, separately. For ASEP, due to \(\ze_1=\om_1+m_1\), there are three cases to consider. Also recall that \(p_0\) is only non-zero at \((0,\,0)\), and \(q_0\) at \((1,\,1)\).
\begin{itemize}
 \item For \(\om_1=m_1=0\), the first half of the display becomes
 \[
  p\vr_0-q\frac{\vr_1}{1-\vr_1}\cdot(1-\vr_0)-p_0(0,\,0)+q_0(1,\,1)\frac{\vr_1}{1-\vr_1}=(p-q)\vr_0-p_0(0,\,0)+q_0(1,\,1)\frac{\vr_0}{1-\vr_0}.
 \]
 \item When \(\om_1=0\), \(m_1=1\), it is
 \[
  -\bigl[(1-\vr_0)q+\vr_0p\bigr]+p\cdot\vr_0+q\cdot(1-\vr_0)=0.
 \]
 \item Finally, when \(\om_1=1\), \(m_1=0\), we have
 \[
  -p\cdot\frac{1-\vr_1}{\vr_1}\cdot\vr_0+q(1-\vr_0)+p_0(0,\,0)\cdot\frac{1-\vr_1}{\vr_1}-q_0(1,\,1)=(q-p)\cdot(1-\vr_0)+p_0(0,\,0)\cdot\frac{1-\vr_0}{\vr_0}-q_0(1,\,1).
 \]
\end{itemize}

Similarly, \(p_K\) requires \((1,\,1)\) and \(q_K\) requires \((0,\,0)\) in order to avoid being zero. The three ASEP cases for the second half that depends on \(\om_K\) are
\begin{itemize}
 \item \(\om_K=m_K=0\),
 \begin{multline*}
  -p\cdot\frac{\vr_K}{1-\vr_K}\cdot(1-\vr_K)+q\vr_K+p_K(1,\,1)\cdot\frac{\vr_K}{1-\vr_K}-q_K(0,\,0)\\
  =(q-p)\vr_K+p_K(1,\,1)\cdot\frac{\vr_K}{1-\vr_K}-q_K(0,\,0);
 \end{multline*}
 \item \(\om_K=0\) and \(m_K=1\), 
 \[
  -\bigl[(1-\vr_K)p+\vr_Kq\bigr]+p\cdot(1-\vr_K)+q\cdot\vr_K=0;
 \]
 \item \(\om_K=1\) and \(m_K=0\),
 \begin{multline*}
  p\cdot(1-\vr_K)-q\cdot\frac{1-\vr_K}{\vr_K}\cdot\vr_K-p_K(1,\,1)+q_K(0,\,0)\cdot\frac{1-\vr_K}{\vr_K}\\
  =(p-q)\cdot(1-\vr_K)-p_K(1,\,1)+q_K(0,\,0)\cdot\frac{1-\vr_K}{\vr_K}.
 \end{multline*}
 The sum of the two halves (the \(\om_1\) and the \(\om_K\) dependent parts) must be zero in all nine cases. The general solution is
 \[
  \begin{aligned}
   p_0(0,\,0)&=(p+c)\vr_0,&\qquad q_0(1,\,1)&=(q+c)(1-\vr_0),\\
   p_K(1,\,1)&=(p+d)(1-\vr_K),&\qquad q_K(0,\,0)&=(q+d)\vr_K
  \end{aligned}
 \]
 for any constants \(c\) and \(d\) in \([-\min(p,\,q),\,\infty)\). This concludes the proof of the theorem for ASEP.
\end{itemize}

Next we turn to BLP for \eqref{eq:reveng}. Recall \eqref{eq:teshi}, \eqref{eq:fprod}, \eqref{eq:bef}, \eqref{eq:pf}, \eqref{eq:ef}, \eqref{eq:blpshrj}, \eqref{eq:blpshrb}. Separating the \(\om_1\) and the \(\om_K\) dependent parts again, the former reads
\[
 \begin{aligned}
  &\quad\e{-\te(\vr_1+m_1)}-\e{-\te(\vr_1)}\\
  &\quad-\sum_{\om_0}\Bigl\{\bigl(f(\om_0)+f(1-\om_1-m_1)\bigr)\cdot\e{-\te_1}f(\om_1)-f(\om_0)-f(-\om_1)\Bigr\}\cdot\mu^{\te_0}(\om_0)\\
  &\quad+p_0(\om_1-1,\,\om_1+m_1-1)\cdot\e{-\te_1}f(\om_1)-p_0(\om_1,\,\om_1+m_1)\\
  &\quad+q_0(\om_1+1,\,\om_1+m_1+1)\cdot\e{\te_1}f(-\om_1)-q_0(\om_1,\,\om_1+m_1)\\
  &=\e{-\be m_1}\e{-\te_1}-\e{-\te_1}\\
  &\quad-\e{\te_0-\te_1}f(\om_1)-\e{-\te_1}f(1-\om_1-m_1)\cdot f(\om_1)+\e{\te_0}+f(-\om_1)\\
  &\quad+p_0(\om_1-1,\,\om_1+m_1-1)\cdot\e{-\te_1}f(\om_1)-p_0(\om_1,\,\om_1+m_1)\\
  &\quad+q_0(\om_1+1,\,\om_1+m_1+1)\cdot\e{\te_1}f(-\om_1)-q_0(\om_1,\,\om_1+m_1)\\
  &=\e{-\be m_1}\e{-\te_1}-\e{-\te_1}\\
  &\quad-\e{\be m_1}f(\om_1)-\e{-\be m_1}\e{-\te_1}+\e{\be m_1}\e{\te_1}+f(-\om_1)\\
  &\quad+p_0(\om_1-1,\,\om_1+m_1-1)\cdot\e{-\te_1}f(\om_1)-p_0(\om_1,\,\om_1+m_1)\\
  &\quad+q_0(\om_1+1,\,\om_1+m_1+1)\cdot\e{\te_1}f(-\om_1)-q_0(\om_1,\,\om_1+m_1)\\
  &=\e{\be m_1}\e{\te_1}-\e{-\te_1}-\e{\be m_1}f(\om_1)+f(-\om_1)\\
  &\quad+p_0(\om_1-1,\,\om_1+m_1-1)\cdot\e{-\te_1}f(\om_1)-p_0(\om_1,\,\om_1+m_1)\\
  &\quad+q_0(\om_1+1,\,\om_1+m_1+1)\cdot\e{\te_1}f(-\om_1)-q_0(\om_1,\,\om_1+m_1).
 \end{aligned}
\]
The second half of \eqref{eq:reveng} is
\[
 \begin{aligned}
  &\quad\e{\te(\vr_K)}-\e{\te(\vr_K+m_K)}\\
  &\quad-\sum_{\om_{K+1}}\Bigl\{\bigl(f(\om_K+1)+f(-\om_{K+1})\bigr)\cdot\e{\te_K}f(-\om_K)-f(\om_K+m_K)-f(-\om_{K+1})\Bigr\}\cdot\mu^{\te_K}(\om_{K+1})\\
  &\quad+p_K(\om_K+1,\,\om_K+m_K+1)\cdot\e{\te_K}f(-\om_K)-p_K(\om_K,\,\om_K+m_K)\\
  &\quad+q_K(\om_K-1,\,\om_K+m_K-1)\cdot\e{-\te_K}f(\om_K)-q_K(\om_K,\,\om_K+m_K)\\
  &=\e{\te_K}-\e{\be m_K}\e{\te_K}\\
  &\quad-\e{\te_K}f(\om_K+1)f(-\om_K)-\e{\te_K}\e{-\te_{K+1}}f(-\om_K)+f(\om_K+m_K)+\e{-\te_K}\\
  &\quad+p_K(\om_K+1,\,\om_K+m_K+1)\cdot\e{\te_K}f(-\om_K)-p_K(\om_K,\,\om_K+m_K)\\
  &\quad+q_K(\om_K-1,\,\om_K+m_K-1)\cdot\e{-\te_K}f(\om_K)-q_K(\om_K,\,\om_K+m_K)\\
  &=\e{\te_K}-\e{\be m_K}\e{\te_K}\\
  &\quad-\e{\te_K}-f(-\om_K)+\e{\be m_K}f(\om_K)+\e{-\te_K}\\
  &\quad+p_K(\om_K+1,\,\om_K+m_K+1)\cdot\e{\te_K}f(-\om_K)-p_K(\om_K,\,\om_K+m_K)\\
  &\quad+q_K(\om_K-1,\,\om_K+m_K-1)\cdot\e{-\te_K}f(\om_K)-q_K(\om_K,\,\om_K+m_K)\\
  &=\e{-\te_K}-\e{\be m_K}\e{\te_K}-f(-\om_K)+\e{\be m_K}f(\om_K)\\
  &\quad+p_K(\om_K+1,\,\om_K+m_K+1)\cdot\e{\te_K}f(-\om_K)-p_K(\om_K,\,\om_K+m_K)\\
  &\quad+q_K(\om_K-1,\,\om_K+m_K-1)\cdot\e{-\te_K}f(\om_K)-q_K(\om_K,\,\om_K+m_K).
 \end{aligned}
\]
Define the function
\[
 a(x,\,y):\,=\Bigl(p_0(x,\,y)-q_0(x+1,\,y+1)\cdot\e{\te_1}f(-x)\Bigr)\mu^{\te_1}(x)
\]
and notice, via \eqref{eq:fprod} and \eqref{eq:ef},
\[
 \begin{aligned}
  &\quad p_0(x-1,\,y-1)\cdot\e{-\te_1}f(x)-p_0(x,\,y)+q_0(x+1,\,y+1)\cdot\e{\te_1}f(-x)-q_0(x,\,y)\\
  &=\Bigl(p_0(x-1,\,y-1)-q_0(x,\,y)\cdot\e{\te_1}f(1-x)\Bigr)\cdot\e{-\te_1}f(x)-\Bigl(p_0(x,\,y)-q_0(x+1,\,y+1)\cdot\e{\te_1}f(-x)\Bigr)\\
  &=a(x-1,\,y-1)\cdot\frac{\e{-\te_1}f(x)}{\mu^{\te_1}(x-1)}-a(x,\,y)\cdot\frac1{\mu^{\te_1}(x)}\\
  &=\bigl(a(x-1,\,y-1)-a(x,\,y)\bigr)\cdot\frac1{\mu^{\te_1}(x)}.
 \end{aligned}
\]
Similarly,
\[
 b(x,\,y):\,=\Bigl(p_K(x,\,y)-q_K(x-1,\,y-1)\cdot\e{-\te_K}f(x)\Bigr)\mu^{\te_K}(x),
\]
then
\[
 \begin{aligned}
  &\quad p_K(x+1,\,y+1)\cdot\e{\te_K}f(-x)-p_K(x,\,y)+q_K(x-1,\,y-1)\cdot\e{-\te_K}f(x)-q_K(x,\,y)\\
  &=\!\Bigl(p_K(x+1,\,y+1)-q_K(x,\,y)\cdot\e{-\te_K}f(1+x)\Bigr)\cdot\e{\te_K}f(-x)-\Bigl(p_K(x,\,y)-q_K(x-1,\,y-1)\cdot\e{-\te_K}f(x)\Bigr)\\
  &=b(x+1,\,y+1)\cdot\frac{\e{\te_K}f(-x)}{\mu^{\te_K}(x+1)}-b(x,\,y)\cdot\frac1{\mu^{\te_K}(x)}\\
  &=\bigl(b(x+1,\,y+1)-b(x,\,y)\bigr)\cdot\frac1{\mu^{\te_K}(x)}.
 \end{aligned}
\]

With these, \eqref{eq:reveng} turns into
\[
 \begin{aligned}
  0&=\e{\be m_1}\e{\te_1}-\e{-\te_1}-\e{\be m_1}f(\om_1)+f(-\om_1)\\
  &\quad+\bigl(a(\om_1-1,\,\om_1+m_1-1)-a(\om_1,\,\om_1+m_1)\bigr)\cdot\frac1{\mu^{\te_1}(\om_1)}\\
  &+\e{-\te_K}-\e{\be m_K}\e{\te_K}-f(-\om_K)+\e{\be m_K}f(\om_K)\\
  &\quad+\bigl(b(\om_K+1,\,\om_K+m_K+1)-b(\om_K,\,\om_K+m_K)\bigr)\cdot\frac1{\mu^{\te_K}(\om_K)}.
 \end{aligned}
\]
Since the first half of this only depends on \(\om_1\) and \(m_1\) and the second half on \(\om_K\) and \(m_K\), the general solution is 
\begin{equation}
 \begin{aligned}
  a(x,\,x+m)-a(x-1,\,x+m-1)&=\Bigl(\e{\be m}\e{\te_1}-\e{-\te_1}-\e{\be m}f(x)+f(-x)+C\Bigr)\cdot\mu^{\te_1}(x),\\
  b(x,\,x+m)-b(x+1,\,x+m+1)&=\Bigl(\e{-\te_K}-\e{\be m}\e{\te_K}-f(-x)+\e{\be m}f(x)-C\bigr)\cdot\mu^{\te_K}(x)
 \end{aligned}\label{eq:abrec}
\end{equation}
for arbitrary \(C\in\Rb\) and all \(x\in\Zb\), \(m\in\Zb_{\ge0}\). We seek boundary rates of finite expectation, hence it is a natural assumption that
\[
 \sum_xp_0(x,\,x+m)\mu^{\te_1}(x),\qquad\sum_xp_K(x,\,x+m)\mu^{\te_K}(x)
\]
as well as
\[
 \begin{aligned}
  \sum_xq_0(x+1,\,x+m+1)\e{\te_1}f(-x)\mu^{\te_1}(x)&=\sum_xq_0(x,\,x+m)\mu^{\te_1}(x)\\
  \sum_xq_K(x-1,\,x+m-1)\e{-\te_K}f(x)\mu^{\te_K}(x)&=\sum_xq_K(x,\,x+m)\mu^{\te_K}(x)
 \end{aligned}
\]
are each finite. Under that assumption summing \eqref{eq:abrec} and using \eqref{eq:ef} yields \(C=0\). Still using summability and \eqref{eq:ef},
\[
 \begin{aligned}
  a(x,\,x+m)&=\sum_{z=-\infty}^x\bigl(a(z,\,z+m)-a(z-1,\,z+m-1)\Bigr)\\
  &=\sum_{z=-\infty}^x\Bigl(\e{\be m}\e{\te_1}\bigl(\mu^{\te_1}(z)-\mu^{\te_1}(z-1)\bigr)+\e{-\te_1}\bigl(\mu^{\te_1}(z+1)-\mu^{\te_1}(z)\bigr)\Bigr)\\
  &=\e{\be m}\e{\te_1}\mu^{\te_1}(x)+\e{-\te_1}\mu^{\te_1}(x+1)\\
  &=\Bigl(\e{\be m}\e{\te_1}+f(-x)\Bigr)\mu^{\te_1}(x);\\
  b(x,\,x+m)&=\sum_{z=x}^\infty\bigl(b(z,\,z+m)-b(z+1,\,z+m+1)\Bigr)\\
  &=\sum_{z=x}^\infty\Bigl(\e{-\te_K}\bigl(\mu^{\te_K}(z)-\mu^{\te_K}(z+1)\bigr)+\e{\be m}\e{\te_K}\bigl(\mu^{\te_K}(z-1)-\mu^{\te_K}(z)\bigr)\Bigr)\\
  &=\e{-\te_K}\mu^{\te_K}(x)+\e{\be m}\e{\te_K}\mu^{\te_K}(x-1)\\
  &=\Bigl(\e{-\te_K}+\e{\be m}f(x)\Bigr)\mu^{\te_K}(x).
 \end{aligned}
\]
We arrive to
\begin{equation}
 \begin{aligned}
  p_0(x,\,x+m)-q_0(x+1,\,x+m+1)\cdot\e{\te_1}f(-x)&=\e{\be m}\e{\te_1}+f(-x),\\
  p_K(x,\,x+m)-q_K(x-1,\,x+m-1)\cdot\e{-\te_K}f(x)&=\e{-\te_K}+\e{\be m}f(x).
 \end{aligned}\label{eq:truth}
\end{equation}
The natural, totally asymmetric choice would be
\[
 \begin{aligned}
  p_0(x,\,x+m)&=\e{\be m}\e{\te_1}+f(-x),&\qquad q_0(x,\,x+m)&\equiv0,\\
  p_K(x,\,x+m)&=\e{-\te_K}+\e{\be m}f(x),&\qquad q_K(x,\,x+m)&\equiv0.
 \end{aligned}
\]
By \(\te_0=\te_1+\be m\) and \(\te_K=\te_{K+1}\), we can interpret these right jumps as
\[
 \begin{aligned}
  p_0(\om_1,\,\om_1+m_1)&=\e{\te_0}+f(-\om_1),\\
  p_K(\om_K,\,\om_K+m_K)&=f(\om_K+m_K)+\e{-\te_{K+1}}
 \end{aligned}
\]
as stated in the theorem.

However, \eqref{eq:truth} allows the addition of reversible boundary dynamics, with arbitrary non-negative \(q^+_0(\om_1,\,\om_1+m_1)\) and \(q^+_K(\om_K,\,\om_K+m_K)\), and
\[
 \begin{aligned}
  p^+_0(x,\,x+m)&=q^+_0(x+1,\,x+m+1)\cdot\e{\te_1}f(-x),\\
  p^+_K(x,\,x+m)&=q^+_K(x-1,\,x+m-1)\cdot\e{-\te_K}f(x).
 \end{aligned}
\]
This completes the proof for BLP.
\end{proof}

\section{Stationary distributions}\label{sc:sd}

We first prove the stationarity results for ASEP.
\begin{proof}[Proof of Proposition \ref{pr:wdbl}]
 We check detailed balance. The state space of the \(D\)'s is that each of them is at least 1 and, given that a move does not lead out of this state space, the following can happen:
 \begin{itemize}
  \item With rate \(P^M\), \(D_0^M=z\) increases by one, and with rate \(Q^M\), it decreases by one. The detailed balance equation
  \[
   P^M\cdot\Bigl(\frac{P^M}{Q^M}\Bigr)^{z-1}=
   Q^M\cdot\Bigl(\frac{P^M}{Q^M}\Bigr)^z
  \]
  holds.
  \item For \(1\le\ell<M\), with rate \(P^\ell\), \(D_0^\ell=z\) increases by one and \(D_0^{\ell+1}=y>2\) decreases by one, and with rate \(Q^\ell\), it is the other way around. The detailed balance equation this time is
  \[
   P^\ell\cdot\Bigl(\prod_{n=\ell}^M\frac{P^n}{Q^n}\Bigr)^{z-1}\Bigl(\prod_{n=\ell+1}^M\frac{P^n}{Q^n}\Bigr)^{y-1}
   =Q^\ell\cdot\Bigl(\prod_{n=\ell}^M\frac{P^n}{Q^n}\Bigr)^z\Bigl(\prod_{n=\ell+1}^M\frac{P^n}{Q^n}\Bigr)^{y-2},
  \]
 which also holds.
 \end{itemize}
\end{proof}
\begin{proof}[Proof of Proposition \ref{pr:rh} for ASEP]
 By the definitions \eqref{eq:ellthr} and \eqref{eq:ellthl},
 \begin{equation}
 \begin{aligned}
  \prod_{n=\ell}^M\frac{P^n}{Q^n}
  =\frac{\prod_{n=\ell}^MP^n}{\prod_{n=\ell}^MQ^n}&=\Bigl(\frac pq\Bigr)^{M-\ell+1}\cdot\frac{1-\vr_0+\bigl(\frac pq\bigr)^{\ell-1}\cdot\vr_0}{1-\vr_0+\bigl(\frac pq\bigr)^M\cdot\vr_0}\cdot\frac{1-\vr_0+\bigl(\frac pq\bigr)^{\ell-1}\cdot\vr_0}{1-\vr_0+\bigl(\frac pq\bigr)^M\cdot\vr_0}\\
  &=\biggl(\frac{\bigl(\frac qp\bigr)^{\frac{\ell-1}2}\cdot(1-\vr_0)+\bigl(\frac pq\bigr)^{\frac{\ell-1}2}\cdot\vr_0}{\bigl(\frac qp\bigr)^{\frac M2}\cdot(1-\vr_0)+\bigl(\frac pq\bigr)^{\frac M2}\cdot\vr_0}\biggr)^2.
 \end{aligned}\label{eq:pqprod}
\end{equation}
As a function of \(\ell\), the numerator is a positive linear combination of two exponentials, hence convex. This shows that the extrema of \eqref{eq:pqprod} are reached at the boundaries \(\ell=1\) or \(\ell=M\). These are
 \[
  \biggl(\frac{\bigl(\frac pq\bigr)^{\frac M2}}{1-\vr_0+\bigl(\frac pq\bigr)^M\cdot\vr_0}\biggr)^2
  \qquad\text{and}\qquad
  \biggl(\frac{\bigl(\frac pq\bigr)^\frac12(1-\vr_0)+\bigl(\frac pq\bigr)^{M-\frac12}\cdot\vr_0}{1-\vr_0+\bigl(\frac pq\bigr)^M\cdot\vr_0}\biggr)^2.
 \]
 We have a stationary probability distribution on the model if and only if both of these quantities is less than 1 i.e., when \(p>q\),
 \[
  \begin{aligned}
   1-\frac{\bigl(\frac pq\bigr)^{\frac M2}}{1-\vr_0+\bigl(\frac pq\bigr)^M\cdot\vr_0}&>0\\
   \Bigl(\frac pq\Bigr)^M\vr_0-\Bigl(\frac pq\Bigr)^{\frac M2}+1-\vr_0&>0\\
   \vr_0&>\frac{\bigl(\frac pq\bigl)^\frac M2-1}{\bigl(\frac pq\bigl)^M-1}=\frac1{\bigl(\frac pq\bigl)^\frac M2+1},
  \end{aligned}
 \]
 and also
 \[
  \begin{aligned}
   1-\frac{\bigl(\frac pq\bigr)^\frac12(1-\vr_0)+\bigl(\frac pq\bigr)^{M-\frac12}\cdot\vr_0}{1-\vr_0+\bigl(\frac pq\bigr)^M\cdot\vr_0}&>0\\
   {1-\vr_0+\Bigl(\frac pq\Bigr)^M\cdot\vr_0}-\Bigl(\frac pq\Bigr)^\frac12(1-\vr_0)-\Bigl(\frac pq\Bigr)^{M-\frac12}\cdot\vr_0&>0\\
   \vr_0&>\frac{\bigl(\frac pq\bigr)^\frac12-1}{\bigl(\frac pq\bigr)^M\Bigl(1-\bigl(\frac pq\bigr)^{-\frac12}\Bigr)+(\frac pq\bigr)^\frac12-1}\\
   &=\frac1{\bigl(\frac pq\bigr)^{M-\frac12}+1}.
  \end{aligned}
 \]
 This latter is weaker due to \(M-\frac12\ge\frac M2\).

 When \(p<q\), the inequalities turn around after dividing by negative numbers in both cases above. As \(\frac pq<1\) this time, 
 \[
  \vr_0<\frac1{\bigl(\frac pq\bigl)^\frac M2+1}
 \]
 wins again.

 The Rankine-Hugoniot velocity \eqref{eq:aseprh} is
 \[
  \begin{aligned}
   (p-q)\cdot(1-\vr_0-\vr_{K+1})&=(p-q)\cdot\Bigl(1-\vr_0-\frac{\bigl(\frac pq\bigr)^M\vr_0}{1-\vr_0+\bigl(\frac pq\bigr)^M\vr_0}\Bigr)\\
   &=(p-q)\cdot\frac{(1-\vr_0)^2-\bigl(\frac pq\bigr)^M\vr_0^2}{1-\vr_0+\bigl(\frac pq\bigr)^M\vr_0}\\
  \end{aligned}
 \]
 When \(p>q\), this is negative exactly when \(\vr_0>\frac1{\bigl(\frac pq\bigl)^\frac M2+1}\), while for \(p<q\) when \(\vr_0<\frac1{\bigl(\frac pq\bigl)^\frac M2+1}\) as required.
\end{proof}

Next, the BLP stationarity proofs.
\begin{proof}[Proof of Proposition \ref{pr:blprev}]
 We check detailed balance on the two possible types of jumps and their reverses. The first one is \(X_0^M\) jumping right from site \(i<K\) to \(i+1\) and back. This means
 \[
  (d^M_0,\,m_i\ge1,\,m_{i+1}=0)\leftrightarrow(d^M_0+1,\,m_i-1\ge0,\,m_{i+1}+1=1).
 \]
 The detailed balance condition amounts to checking
 \begin{multline*}
  \e{\te_0-\be M}\Bigl(\e{\be m_i}-1\Bigr)\cdot e^{\bigl(\be(M^2-2M+1-M^2)+2\te_0\cdot(M-M+1)\bigr)d^M_0}\cdot\prod_{n=1}^{m_i}\frac1{\e{\be n}-1}\\
  =\e{-\te_0+\be(M-1)}\Bigl(\e{\be}-1\Bigr)\cdot e^{\bigl(\be(M^2-2M+1-M^2)+2\te_0\cdot(M-M+1)\bigr)(d^M_0+1)}\cdot\prod_{n=1}^{m_i-1}\frac1{\e{\be n}-1}\cdot\frac1{\e\be-1}.
 \end{multline*}
 The \((\e\cdot-1)\) factors work out, and we are left to check
 \[
  \te_0-\be M=-\te_0+\be(M-1)-\be(2M-1)+2\te_0
 \]
 which holds.

 The second type of jumps is \(X_0^\ell\), one of \(1\le\ell<M\), jumping right from site \(i<K\) to \(i+1\) and back. This means
 \[
  (d^\ell_0,\,d^{\ell+1}_0\ge1,\,m_i\ge1,\,m_{i+1})\leftrightarrow(d^\ell_0+1,\,d^{\ell+1}_0-1,\,m_i-1,\,m_{i+1}+1).
 \]
 We need to check
 \[
  \begin{aligned}
   \e{\te_0-\be\ell}\Bigl(\e{\be m_i}-1\Bigr)
   \cdot&e^{\bigl(\be(\ell^2-2\ell+1-M^2)+2\te_0\cdot(M-\ell+1)\bigr)d^\ell_0}\\
   \cdot&e^{\bigl(\be((\ell+1)^2-2(\ell+1)+1-M^2)+2\te_0\cdot(M-(\ell+1)+1)\bigr)d^{\ell+1}_0}\\
   \cdot&\prod_{n=1}^{m_i}\frac1{\e{\be n}-1}
   \cdot\prod_{r=1}^{m_{i+1}}\frac1{\e{\be r}-1}\\
   =\e{-\te_0+\be(\ell-1)}\Bigl(\e{\be(m_{i+1}+1)}-1\Bigr)
   \cdot&e^{\bigl(\be(\ell^2-2\ell+1-M^2)+2\te_0\cdot(M-\ell+1)\bigr)(d^\ell_0+1)}\\
   \cdot&e^{\bigl(\be((\ell+1)^2-2(\ell+1)+1-M^2)+2\te_0\cdot(M-(\ell+1)+1)\bigr)(d^{\ell+1}_0-1)}\\
   \cdot&\prod_{n=1}^{m_i-1}\frac1{\e{\be n}-1}
   \cdot\prod_{r=1}^{m_{i+1}+1}\frac1{\e{\be r}-1}.
  \end{aligned}
 \]
 This is equivalent to
 \begin{multline*}
  \te_0-\be\ell\\
  =-\te_0+\be(\ell-1)+\be(\ell^2-2\ell+1-M^2)+2\te_0\cdot(M-\ell+1)-\be\bigl((\ell+1)^2-2(\ell+1)+1-M^2\bigr)-2\te_0\cdot\bigl(M-(\ell+1)+1\bigr)
 \end{multline*}
 which is readily checked.
\end{proof}
\begin{proof}[Proof of Proposition \ref{pr:rh} for BLP]
 Since \(M<\infty\) is fixed, the second factor with the \(m_i\)'s of the measures \eqref{eq:blpstati} is always finite. The measures are normalisable if and only if for each \(1\le\ell\le M\),
 \[
  \be(\ell^2-2\ell+1-M^2)+2\te_0\cdot(M-\ell+1)<0.
 \]
 As \(\ell\le M\), this is a decreasing statement in \(\te_0\), and the first \(\te_0\) where it fails is
 \[
  \frac{\be\bigl(M^2-(\ell-1)^2\bigr)}{2\bigl(M-(\ell-1)\bigr)}=\frac\be2(M+\ell-1).
 \]
 The first \(\ell\) where \(\te_0\) becomes too large is \(\ell=1\), and the normalisation works up to \(\te_0<\frac\be2M\). This is exactly the region where the Rankine-Hugoniot velocity \eqref{eq:blprh} is negative.
\end{proof}

\section*{Acknowledgements}

The author thanks Lewis Duffy, Dimitri Pantelli, and Gunter Sch\"utz for early discussions about the topics of this manuscript.

\hop
This study did not involve any underlying data.

\end{document}